\documentclass[12pt]{amsart}
\usepackage{amscd,amssymb}
\usepackage{graphicx}
\usepackage{color}
\usepackage[arrow,dvips,matrix,arrow,ps,color,line,curve,frame]{xy}
\usepackage{pgf, tikz}
\usetikzlibrary{matrix,arrows}
\usepackage{url}
\usepackage{enumerate}

\makeatletter
\def\url@leostyle{%
  \@ifundefined{selectfont}{\def\UrlFont{\sf}}{\def\UrlFont{\small\ttfamily}}}
\makeatother
\usetikzlibrary{matrix,arrows}

\let\oldlabel=\label

\def\prellabel{\marginparsep=1em
    \def\label##1{\oldlabel{##1}\ifmmode\else\ifinner\else
         \marginpar{{\footnotesize\ \\ \tt
                    ##1}}\fi\fi}}

\def\codim{\operatorname{codim}}
\def\rank{\operatorname{rank}}

\def\Aut{\operatorname{Aut}}

\def\o{\operatorname{o}}

\def\int{{\operatorname{int}}}
\def\conv{\operatorname{conv}}

\def\BC{\operatorname{BC}}

\def\V{\operatorname{V}}

\def\lin{\operatorname{lin}}
\def\vertex{\operatorname{vert}}
\def\rank{\operatorname{rank}}
\def\codim{\operatorname{codim}}
\def\Im{\operatorname{Im}}
\def\Hom{\operatorname{Hom}}

\def\Pol{\mathbf{Pol}}

\def\surj{\operatorname{surj}}
\def\inj{\operatorname{inj}}

\def\RR{{\mathbb R}}

\def\ZZ{{\mathbb Z}}
\def\NN{{\mathbb N}}

\def\FF{{\mathbb F}}

\def\cN{{\mathcal N}}

\def\Aff{\operatorname{Aff}}

\let\epsilon=\varepsilon
\let\phi=\varphi
\let\theta=\vartheta

\advance\headheight1.15pt

\newtheorem{lemma}{Lemma}[section]
\newtheorem{corollary}[lemma]{Corollary}
\newtheorem{theorem}[lemma]{Theorem}
\newtheorem{proposition}[lemma]{Proposition}

\theoremstyle{definition}

\begin{document}

\title[Vertex maps between $\triangle$, $\Box$, and $\Diamond$]{Vertex maps between $\triangle$, $\Box$, and $\Diamond$}

\author[J. Gubeladze]{Joseph Gubeladze}
\author[J. Love]{Jack Love}

\address{Department of Mathematics\\
         San Francisco State University\\
         1600 Holloway Ave.\\
         San Francisco, CA 94132, USA}
\email{soso@sfsu.edu; jlove@mail.sfsu.edu}

\thanks{Joseph Gubeladze was supported by NSF grants DMS-1000641 \& DMS-1301487}


\maketitle

\begin{abstract}
We study the vertices of the polytopes of all affine maps (a.k.a. \emph{hom-polytopes}) between higher dimensional simplices, cubes, and crosspolytopes. Systematic study of general hom-polytopes was initiated in \cite{HOM}. The study of such vertices is the classical aspect of a conjectural homological theory of convex polytopes. One quickly runs into open problems even for simple source and target polytopes.  The vertices of $\Hom(\triangle_m,-)$ and $\Hom(-,\Box_n)$ are easily understood. In this work we describe the vertex sets of $\Hom(\Box_m,\triangle_n)$, $\Hom(\Diamond_m,\triangle_n)$, and $\Hom(\Diamond_m,\Diamond_n)$. The emergent pattern in our arguments is reminiscent of diagram chasing in homological algebra.
\end{abstract}

\section{Introduction}\label{Intro}

Convex polytopes serve as the main vehicle for a major part of algebraic combinatorics. A big part of the theory of convex polytopes studies affine properties of these objects, i.e., the properties that are invariant under affine transformations, as opposed to other properties such as projective, metric, discrete etc. In their turn, affine maps are for affine spaces (e.g., the affine hulls of polytopes) what linear maps are for vector spaces. Thus, on the one hand, the category $\Pol$ of convex polytopes and their affine maps is a natural habitat for polytopal combinatorics and, on the other hand, it resembles the linear category of finite dimensional vector spaces. The latter analogy can be promoted to the following semi-folklore fact: $\Pol$ enjoys a symmetric closed monoidal category structure, enriched on itself. In other words, the set of affine maps between two polytopes forms a polytope in its own right and there is another functorial construction -- the tensor product of polytopes -- satisfying the usual (right) conjunction with the hom-construction \cite{HOM,Valby}.

The importance of the basic fact that the hom-objects in $\Pol$ are polytopes is emphasized in the last pages of \cite{ZiPOL} and the well known software package \textsf{Polymake} \cite{GaJoPOLY} even has a special module to actually compute these objects in terms the source and target polytopes. Although \cite{GaJoPOLY} uses the name \emph{mapping polytopes}, our terminology of \emph{hom-polytopes} is more in line with the categorial point of view. The categorial perspective also suggests what the next natural steps in the process of fusing the polytopal and linear worlds should be. For instance, can one view the Sturmfels-Billera fiber polytopes \cite{BiStuFIBER}, which plays the central role in the theory of regular triangulations, as certain kernel objects in $\Pol$? More interestingly, is there a framework for the still elusive dual \emph{quotient} polytopes? These and other homological polytopal constructions still being crystallized, in this paper we focus on the basic challenge of determination of the vertices of $\Hom(P,Q)$ for classical $P$ and $Q$.

The first substantial treatment of hom-polytopes was given in \cite{HOM}, where basic properties were established. In particular, the paper \cite{HOM} emphasized on the importance of computing the vertices of hom-polytopes -- it was shown that this poses a serious problem even for supposedly tame polytopes (e.g., polygons) as the source and target objects. In this paper we continue the investigation of hom-polytopes along `simple' examples: the simplices, cubes, and crosspolytopes in arbitrary dimension. The emergent rich combinatorics, resulting from the categorial approach, is remarkable. But one can also trace patterns reminiscent of diagram chasing in homological algebra.

Our arguments involve phenomena in polytopes which are often observed for general polytopes and not just for the mentioned class. The most general principle employed in this paper is the following simple \emph{perturbation criterion}: an affine map $f:P\to Q$ is not a vertex of $\Hom(P,Q)$ if and only if  there is a family of affine maps $\{f_t:P\to Q\}_{t\in(-1,1)}$, \emph{smoothly} parameterized by $t$ so that $f_0=f$. Examples of involvement of general polytopes are Lemma \ref{newlemma} and  Theorem \ref{diamondtosimplex}(a).

Before describing the main results we recall the following well known fact (\cite[Section 2]{HOM},\cite[Section 9.4]{ZiPOL}). It explains why the determination of vertices of hom-polytopes is the first step in understanding the hom-polytopes:

\begin{theorem}\label{folklore}
Let $P\subset V$ and $Q\subset W$ be polytopes in their ambient vector spaces.
\begin{itemize}
\item[(a)] The set $\Hom(P,Q)$ naturally embeds as a polytope into the vector space of linear maps $\Hom(W\times\RR,V\times\RR)$.
\item[(b)] The facets of $\Hom(P,Q)$ are the subsets of the form
$$
H(v,F)=\{f\in\Hom(P,Q\}\ |\ f(v)\in F\},
$$
where $v\in P$ is a vertex and $F\subset Q$ is a codimension one face.
\item[(c)] $\dim(\Hom(P,Q))=\dim P\dim Q+\dim Q$.
\item[(d)] For every vertex $w\in Q$, the map $f:P\to Q$, $\Im f=\{w\}$, is a vertex of $\Hom(P,Q)$.
\end{itemize}
\end{theorem}

\medskip The main results in this paper, derived from explicit smooth perturbations of affine maps, are as follows.

\smallskip\noindent$\centerdot$ In Section \ref{cubetosimplex} we show that every vertex $f\in\Hom(\Box_m,\triangle_n)$ maps the $m$-cube $\Box_m$ onto either a vertex or an edge of the $n$-simplex $\triangle_n$. In particular, $\Hom(\Box_m,\triangle_n)$ has $(n+1)(mn+1)$ vertices.

\smallskip\noindent$\centerdot$ In Section \ref{Diamondtosimplex} it is shown that, for the $m$-dimensional crosspolytope $\Diamond_m$ and an $n$-dimensional polytope $P$, every vertex $f\in\Hom(\Diamond_m,P)$ with $\dim(\Im f)=n$ admits an $n$-dimensional sub-crosspolytope $\Diamond\subset\Diamond_m$ such that $f|_{\Diamond}$ is a vertex of $\Hom(\Diamond,P)$. Furthermore, we have a complete geometric description of the vertices $f\in\Hom(\Diamond_m,\triangle_n)$.  These results lead to nontrivial lower and upper estimates for the number of vertices of $\Hom(\Diamond_m,\triangle_n)$.

\smallskip\noindent$\centerdot$ Section \ref{Diamondtodiamond} focuses on the polytope $\Hom(\Diamond_m,\Diamond_n)$. In view of Section \ref{Diamondtosimplex}, the only new situation arises when the image of a vertex map $f\in\Hom(\Diamond_m,\Diamond_n)$ meets the interior of $\Diamond_n$. The main result here is that every such $f$ maps the center of $\Diamond_m$ to that of $\Diamond_n$. As a consequence, we obtain nontrivial estimates for the number of vertices of $\Hom(\Diamond_m,\Diamond_n)$.

\smallskip Since the functors $\Hom(\triangle_m,-)$ and $\Hom(-,\Box_n)$ are well understood (Proposition \ref{calculus}), the only remaining open case is $\Hom(\Box_m,\Diamond_n)$. Our \textsf{Polymake} computations show that an easy lower bound for the number of vertices of $\Hom(\Box_m,\triangle_n)$ is far from being optimal (Section \ref{Boxtodiamond}).

Based on this work one can speculate that among the regular polytopes, the Platonic solids and regular 4-polytopes are the most challenging source/target polytopes for determination of the vertex maps.

\section{Preliminaries}\label{Preliminaries}

Our references on general convex polytopes are  \cite[Ch.1]{KRIPO} and \cite{ZiPOL}. Yet for the reader's convenience below we encapsulate some basic definitions and fix notation, which are not always identical in the two sources.

\subsection{Affine spaces}\label{affinespaces} All our vector spaces are real and finite dimensional. An \emph{affine space} is a translate of a vector subsapcce, i.e., a subset of the form $H=x+V'\subset V$, where $V$ is a vector space, $V'\subset V$ is a subspace, and $x\in V$. A map between two affine spaces $f:H_1\to H_2$ is an \emph{affine map} if $f$ respects barycentric coordinates or, equivalently, $f$ maps polytopes to polytopes and parallel affine subspaces to parallel subspaces, possibly of lesser dimension.

\medskip For a subset $X\subset\RR^n$, denote by:

\noindent $\centerdot$ $\RR X$ the linear span of $X$,

\noindent $\centerdot$ $\conv(X)$ the convex hull of $X$,

\noindent $\centerdot$ $\Aff(X)$ the affine hull of $X$,

\noindent $\centerdot$ $\lin(X)$ the homogenization of $\Aff(X)$, i.e.,  the linear subspace $\Aff(X)-x\subset\RR^n$ where $x$ is some (equivalently, any) element $x\in X$.

\medskip The set of affine maps $H_1\to H_2$ between two affine spaces will be denoted by $\Aff(H_1,H_2)$. We put $\rank f=\dim f(H_1)$.

Let $H_1\subset V_1$ and $H_2\subset V_2$ be affine subspaces in their ambient vector spaces. Upon fixing an affine surjective map $\pi:V_1\to H_1$, which restricts to the identity map on $H_1$, we get an injective map
$$
\theta_\pi:\Aff(H_1,H_2)\to\Aff(V_1,V_2),\quad f\mapsto \iota\circ f\circ\pi,
$$
where $\iota:H_2\hookrightarrow V_2$ is the inclusion map. We also have the embedding into the space of linear maps:
\begin{align*}
\theta:\Aff(V_1,V_2)&\to\Hom(V_1\times\RR,V_2\times\RR),\\
&(\theta(f))(x,c)=(cf(c^{-1}x),c),\quad(\theta(f))(x,0)=(x,0),\\
&\qquad\qquad\qquad\qquad\qquad\qquad\qquad\quad x\in V_1,\quad c\in\RR\setminus\{0\}.\\
\end{align*}
The composite map $\theta\circ\theta_\pi$ identifies $\Aff(H_1,H_2)$ with an affine subspace of $\Hom(V_1\times\RR,V_2\times\RR)$ and the induced convexity notion in $\Aff(H_1,H_2)$ is independent of the choice of $\pi$. This affine embedding is implicit in Theorem \ref{folklore}(a).

Any full rank subset of a vector space contains a basis. By dualizing we get

\begin{lemma}\label{helly}
For two natural numbers $n\le m$ and a system of linear subspaces $V_1,\ldots,V_m\subset\RR^n$ we have the equivalence
\begin{align*}
V_1\cap\cdots\cap V_m=0\ \Longleftrightarrow\ \exists\{i_1,\ldots,i_n\}\subset\{1,\ldots,m\}\ \ V_{i_1}\cap\cdots\cap V_{i_n}=0.
\end{align*}
\end{lemma}

The standard basis vectors of $\RR^n$ are denoted by $e_1,\ldots,e_n$.

For the dual spaces we will make the identification $(\RR^n)^{\o}=\RR^n$, so that the pairing  $(\RR^n)^{\o}\times\RR^n\to\RR$ becomes the dot-product
$(x_1,\ldots,x_n)\cdot(y_1,\ldots,y_n)=x_1y_1+\cdots+x_ny_n$.

\subsection{Polytopes} A \emph{polytope} always means a convex polytope in an ambient vector (or affine) space. An \emph{affine map} between two polytopes is the restriction of an affine map between the ambient spaces. We write $P\cong Q$ if $P$ and $Q$ are isomorphic polytopes in $\Pol$.

A \emph{face projection} of a polytope $P$ refers to a surjective map in $\Pol$, representing a parallel projection from $P$ along the affine hull of a face of $P$.

\medskip For two polytopes $P$ and $Q$ let:

\noindent $\centerdot$ $\Aff(P,Q):=\Aff(\Hom(P,Q))\quad(=\Aff(\Aff(P),\Aff(Q)))$;

\noindent $\centerdot$ an element $f\in\Hom(P,Q)$ be called a \emph{vertex map} if $f$ is a vertex of $\Hom(P,Q)$;

\noindent $\centerdot$ $\rank f$ denote the rank of the affine extension of $f$ to  $\Aff(P,Q)$;

\noindent $\centerdot$ $\vertex(P,Q)$ denote the set of vertices of $\Hom(P,Q)$;\footnote{In \cite{HOM} the vertex set is denoted by $\vertex(\Hom(P,Q))$.}

\noindent $\centerdot$ $\vertex^{(k)}(P,Q)$ denote the set of rank $k$ elements of $\vertex(P,Q)$;

\noindent $\centerdot$ $\FF(P)$ the set of \emph{facets} (i.e., the faces of dimension $\dim P-1$) of $P$;

\noindent $\centerdot$ $\int(P)$ the interior of $P$ (relative to $\Aff(P)$);

\noindent $\centerdot$ $\partial P=P\setminus\int(P)$ -- the boundary of $P$.

\medskip If, additionally, $P\subset\RR^n$ and $0\in\int(P)$, denote by:

\noindent $\centerdot$ $P^{\o}$ the \emph{dual polyhedron}, i.e., $P^{\o}=\{x\in\RR^n\ |\ x\cdot y\le1\ \text{for all}\ y\in P\}$; $P^{\o}$ is a polytope iff $\dim P=n$; in that case $\dim P^{\o}=n$ (\cite[Section 1.B]{KRIPO},\cite[Section 2.3]{ZiPOL});

\noindent $\centerdot$ $\Diamond(P)$ the bi-pyramid $\conv(P,\pm e_{n+1})\subset\RR^{n+1}$, where $\RR^n$ is viewed as the hyperplane in $\RR^{n+1}$ of the first $n$-coordinates.

\medskip The standard $n$-simplex, $n$-cube, and $n$-crosspolytope are defined as follows:
\begin{align*}
&\triangle_n=\conv(0,e_1,\ldots,e_n),\\
&\Box_n=\conv\big((a_1,\ldots,a_n)\ |\ a_i=\pm1\big),\\
&\Diamond_n=\conv(\pm e_1,\ldots,\pm e_n).
\end{align*}

In particular, $\Box_n$ and $\Diamond_n$ are dual to each other.

\medskip For a polytope $P$, the group of its automorphisms in $\Pol$ will be denoted by $\Aut(P)$. The $n$-th \emph{hyperoctahedral group} is
$$
\BC_n=\Aut(\Box_n)=\Aut(\Diamond_n).
$$
One has $|\BC_n|=2^nn!$ and the center of $\BC_n$ is isomorphic to $(\ZZ_2)^n$ (\cite[Section 7.6]{Coxeter}). Also, we think of $\Aut(\triangle_n)$ as the $(n+1)$-st permutation group $S_{n+1}$.

\subsection{Cones} For generalities on cones see \cite[Chapter 1]{KRIPO} and \cite[Chapter 1]{ZiPOL}.

The set of nonnegative reals is denoted by $\RR_+$. For a subset $X\subset\RR^n$, the set $\RR_+ X$ (the set of non-negative real linear combinations of finitely many elements of $X$) is called the \emph{conical hull} of $X$. A \emph{cone} means a finite polyhedral pointed cone, i.e., a subset of the form $C=\RR_+x_1+\cdots+\RR_+x_k\subset\RR^n$ for some finite set $\{x_1,\ldots,x_n\}\subset\RR^n$ and containing no non-zero linear subspace. The \emph{dual conical set} $C^{\o}=\{x\in\RR^n\ |\ x\cdot y\ge0\ \text{for all}\ y\in C\}\subset\RR^n$ is a cone iff $\dim C=n$. In that case $C^{\o}$ is called the \emph{dual cone} for $C$.

Let $P\subset\RR^n$ be a polytope and $v\in P$ a vertex. The \emph{affine corner cone of $P$ at $v$} is the shifted cone $v+\RR_+(P-v)$. If $\dim P=n$ then the \emph{normal fan of $P$} is the fan $\cN(P)$ in $\RR^n$ whose maximal cones are the duals of the \emph{corner cones} of $P$, i.e., the maximal cones of $\cN(P)$ are of the form $(\RR_+(P-v))^{\o}$, where $v$ runs over $\vertex(P)$. For generalities on normal fans see \cite[Capter 1]{KRIPO}.

For two $n$-dimensional polytopes $P,Q\subset\RR^n$, we have $\cN(P)=\cN(Q)$ if and only if $P$ and $Q$ are of same combinatorial type and the corresponding corner cones are equal if and only if $P$ and $Q$ have same combinatorial types and the corresponding faces are parallel.

\subsection{Some facts on affine maps}

Every map $f:P\to Q$ in $\Pol$ factors as follows
\[
\begin{tikzpicture}[description/.style={fill=white,inner sep=2pt}]
\matrix (m) [matrix of math nodes, row sep=2em,
column sep=2em, text height=1.5ex, text depth=0.25ex]
{ P & & Q \\
& R & \\ };
\path[->,font=\scriptsize]
(m-1-1) edge node[auto] {$ f $} (m-1-3)
edge node[auto, swap] {$ g $} (m-2-2)
(m-2-2) edge node[auto, swap] {$ h $} (m-1-3);
\end{tikzpicture}
\]
where $g$ is surjective and $h$ is injective. Moreover, any two such factorizations fit into a unique commutative diagram:

\[
\begin{tikzpicture}[description/.style={fill=white,inner sep=1.5pt}]
\matrix (m) [matrix of math nodes, row sep=2em,
column sep=2em, text height=1.5ex, text depth=0.25ex]
{ & R & \\
 P & & Q \\
& R' & \\ };
\path[->,font=\scriptsize]
(m-2-1) edge node[auto, swap] {} (m-3-2)
edge node[auto] {} (m-1-2)
(m-1-2) edge node[auto, swap] {$\cong$} (m-3-2)
edge node[auto] {} (m-2-3)
(m-3-2) edge node[auto, swap] {} (m-2-3);
\end{tikzpicture}
\]

We let $f_{\surj}$ and $f_{\inj}$ denote any representative of such factorizations: $f=f_{\inj}\circ f_{\surj}$.

\begin{proposition}\label{calculus1} Let $f:P\to Q$ be an affine map of polytopes.
\begin{itemize}
\item[(a)] If $f(\vertex(P))\subset\vertex(Q)$ then $f\in\vertex(P,Q)$.
\item[(b)] If $f\in\vertex(P,Q)$ then $f_{\surj}\in\vertex(P,\Im f)$ and $f_{\inj}\in\vertex(\Im f,Q)$.
\item[(c)] Assume $\rank f=1$. Then $f\in\vertex(P,Q)$ if and only if $f$ is a facet projection of $P$ onto an edge or a diagonal of $Q$.
\end{itemize}
\end{proposition}

\begin{proof}
(a) follows from the perturbation criterion $(\text{pc}_1)$ in \cite[Section 4]{HOM}. (See Lemma \ref{perturbation} below.)

(b) is \cite[Theorem 4.1(1)]{HOM}.

(c) follows from \cite[Corollary 4.3]{HOM}
\end{proof}

Examples showing that the implication (b) can not be reversed are given in \cite[Section 5]{HOM}.

\begin{proposition}\label{calculus} Let $n$ be a natural numbers and $P$ be a polytope.
\begin{itemize}
\item[(a)] $\Hom(\triangle_n,P)\cong P^{n+1}$.
\item[(b)] If $P\subset\RR^m$, $\dim P=m$ and $P$ is centrally symmetric w.r.t. the origin then $\Hom(P,\Box_n)\cong\Diamond(P^{\o})^n$.
\end{itemize}
\end{proposition}

This is proved in \cite[Corollary 3.6]{HOM}.

As a special case of Proposition \ref{calculus}(b), we have $\Hom(\Box_m,\Box_n)\cong(\Diamond_{m+1})^n$ and $\Hom(\Diamond_m,\Box_n)\cong\Diamond(\Box_m)^n$. For $P$ not necessarily centrally symmetric, the polytope $\Hom(P,\Box_n)$ is still easily described, but the information on the face lattice structure is not as straightforward (\cite[Proposition 3.2(6)]{HOM}).

\section{Smooth perturbation criterion}\label{Smooth}

A \emph{smooth (affine) 1-family} in $\Hom(P,Q)$ is a subset of the form
$$
\{f_t\}_{(-1,1)}\subset\Hom(P,Q),
$$
where the map $\psi:(-1,1)\to\Aff(P,Q)$, determined by $\psi(t)|_P=f_t$, is injective and smooth (respectively, affine).  A \emph{smooth (affine) perturbation} of an element $f\in\Hom(P,Q)$ is a smooth (respectively, affine) 1-family $\{f_t\}_{(-1,1)}\subset\Hom(P,Q)$ with $f_0=f$.

The following criterion for vertex maps is just a specialization of the fact that a point in a polytope is not a vertex if and only if there is a smooth curve inside the polytope, passing through the point, if and only if there is an open interval inside the polytope, passing through the point.

\begin{lemma}[{Perturbation criterion}]\label{perturbation} In the notation above,
$f\notin\vertex(P,Q)$ if and only if $f$ admits a smooth perturbation if and only if $f$ admits an affine perturbation.
\end{lemma}

In \cite{HOM} only the affine perturbation criterion is used. In this paper we will need both criteria. The smooth version is used in Lemmas \ref{insidesimplex} and \ref{newlemma}.

\begin{corollary}\label{convexposition1}
Let $P, R$, and $Q$ be polytopes. Assume $R\subset Q$ and $m\ge2$.
\begin{itemize}
\item[(a)] If $f\in\vertex(P,Q)$ and $f(P)\subset R$ then $f\in\vertex(P,R)$.
\item[(b)] For $P$ centrally symmetric we have the implication
\begin{align*}
f\in\vertex(\Diamond_m,P)\ \&\ (f(0)\ \text{is the center of}&\ P)\quad\Longrightarrow\\
&f(\vertex(\Diamond_m))\subset\vertex(P).
\end{align*}
\item[(c)] If $f\in\vertex(\Diamond_m,P)$ then
$$
f(\vertex(\Diamond_m))=\vertex(\Im f)\subset\vertex(P\cap(2f(0)-P)).
$$
\end{itemize}
\end{corollary}

\begin{proof} (a) A smooth 1-family in $\Hom(P,Q)$ is a smooth 1-family in $\Hom(P,R)$.

\medskip(b) Assume $f(v)\notin\vertex(P)$ for some $v\in\vertex(\Diamond_m)$ and fix an open interval $(x,y)\subset P$ centered at $f(v)$. Then the family $\{f_t\}_{(-1,1)}\subset\Hom(\Diamond_m,P)$, defined by
\begin{align*}
f_t(w)=
\begin{cases}
\big((1-t)x+(1+t)y\big)/2,\ \text{if}\ w=v,\\
\big((1-t)(2f(0)-x)+(1+t)(2f(0)-y)\big)/2,\ \text{if}\ w=f(-v),\\
f(w),\ \text{if}\ w\in\vertex(\Diamond_m)\setminus\{v,-v\},
\end{cases}
\end{align*}
is an affine 1-family, contradicting Lemma \ref{perturbation}.

\medskip(c) By Proposition \ref{calculus1}(b), $f\in\vertex(\Diamond_m,\Im f)$. So the equality follows from part (b). For the inclusion, we observe that $\Im f\subset P\cap(2f(0)-P)$ (just because $\Im f$ is centrally symmetric w.r.t. $f(0)$). So (b) applies again.
\end{proof}

\begin{lemma}\label{insidesimplex}
Let $n$ be a natural number, $P$ a polytope, and $f\in\vertex(P,\triangle_n)$. Then there exists a face $G\subset\triangle_n$ such that $\dim f(P)=\dim G$, $f(P)\subset G$, and every facet of $G$ contains a facet of $f(P)$.
\end{lemma}

In the proof we will use rotations around codimension 2 affine subspaces. Let $n\ge2$ and $H\subset\RR^n$ be a codimension two affine subspace. A \emph{rotation around $H$} is an affine automorphism of the form
$$
\rho^H:\RR^n\to\RR^n,\qquad x\mapsto \rho(x-h)+h,
$$
where $h\in H$ and $\rho\in SO(n)$ is an element which fixes $\lin(H)$. We have the bijective correspondence $\rho^H\leftrightarrow \rho$. Consequently, the rotations around $H$ are naturally parameterized by the unit circle $S^1$ and we get the embedding of Lie groups:
\begin{align*}
S^1\to SO(n),\quad t\mapsto\rho^H_t\mapsto\rho_t,\qquad t\in S^1.
\end{align*}
($S^1$ is thought of as the multiplicative group of complex numbers of norm 1.)

\begin{proof}[Proof of Lemma \ref{insidesimplex}]
If $P\subset\partial\triangle_n$ then induction on dimension applies in view of Corollary \ref{convexposition1}(a). So there is no loss of generality in assuming $f(P)\cap\int(\triangle_n)\not=\emptyset$.

We want to show $\dim(F\cap f(P))=n-1$ for all $F\in\FF(\triangle_n)$, which also implies $\dim f(P)=n$.

Assume, to the contrary, $\dim(F\cap f(P))\le n-2$ for a facet $F\subset\triangle_n$.

Without loss of generality we can further assume $F=\conv(e_1,\ldots,e_n)\in\FF(\triangle_n)$.

Let $H_F=\Aff(F)$ and $H_F^+\subset\RR^n$ be the affine half-space containing $\triangle_n$.

Let $H\subset H_F$ be an $(n-2)$-dimensional affine subspace, containing $f(P)\cap F$. There exists a small open arc $\Gamma\subset S^1$, containing $1$, such that
\begin{align*}
f(P)\subset \rho^H_t(H_F^+),\quad t\in\Gamma.
\end{align*}
Consider the simplices
$$
\triangle_t=\RR^n_+\cap\rho^H_t(H_F^+),\quad t\in\Gamma.
$$
We have
$$
\triangle_t=\conv(0,\lambda_{t1}e_1,\ldots,\lambda_{tn}e_n)
$$
for some real numbers $\lambda_{t1},\ldots,\lambda_{tn}>0$. Furthermore, the maps
$$
\Gamma\to\RR,\quad t\mapsto\lambda_{ti},\qquad i=1,\ldots,n,
$$
are smooth and at least one of them is injective. Therefore, the linear maps
$$
\alpha_t:\RR^n\to\RR^n,\quad\alpha_t(e_i)=\lambda^{-1}_{ti}e_i,\qquad\qquad i=1,\ldots,n,
$$
give rise to a smooth embedding $\Gamma\to GL_n(\RR)$. Moreover, for any point $x\in\int(\triangle_n)$, the map
$$
\Gamma\to\RR_+^n,\quad t\mapsto \alpha_t(x),
$$
is smooth and \emph{injective}.

Summing up, we have:

\noindent$\centerdot$ $\alpha_t(f(P))\subset\triangle_n$ for every $t\in\Gamma$,

\noindent$\centerdot$ the map $\Gamma\to\Hom(P,\triangle_n)\hookrightarrow\Aff(P,\triangle_n)$, $t\mapsto\alpha_t\circ f$,
is smooth injective.

\smallskip Straightening out $\Gamma$ via a diffeomorphism $\Gamma\to(-1,1)$, which maps $1$ to $0$, we get a smooth perturbation of $f$, and this contradicts Lemma \ref{perturbation}.

Figure 1 represents the face $F$ being rotated about the codimension 2 affine space containing $\Aff(f(P)\cap F)$.

\begin{figure}[htb]
\includegraphics[scale=.30]{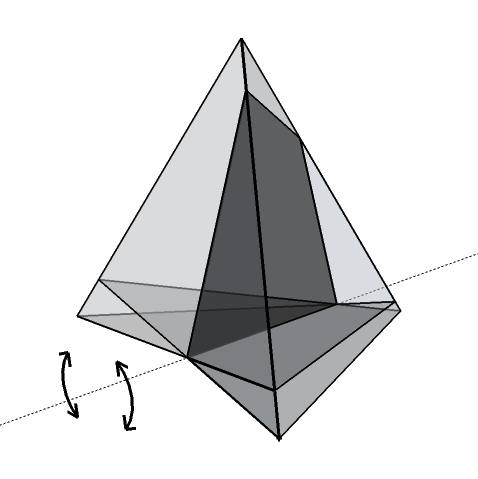}
\caption{}
\end{figure}
\end{proof}

\section{$\vertex(\Box_m,\triangle_n)$}\label{cubetosimplex}

The main result in this section is that every vertex map $\Box_m\to\triangle_n$ collapses the cube into either a vertex or an edge of $\triangle_n$. This fact allows a particular realization of the set $\vertex(\Box_m,\triangle_n)\subset\RR^{mn+n}$. In the proof of the main result we seek a smooth perturbation of an arbitrary rank $k\ge2$ element of
$\Hom(\Box_m,\triangle_n)$. A dualization argument allows us to use special perturbations of simplices in crosspolytopes, constructed in Lemma \ref{simplexindiamond}.

\begin{theorem}\label{boxtriangle}
For all natural numbers $m$ and $n$ we have
$$
\vertex(\Box_m,\triangle_n)=\vertex^{(0)}(\Box_m,\triangle_n)\ \cup\ \vertex^{(1)}(\Box_m,\triangle_n).
$$
\end{theorem}

\begin{corollary}\label{BT}
The polytope $\Hom(\Box_m,\triangle_n)$ has $(n+1)(mn+1)$ vertices. The following set is a particular realization of $\vertex(\Box_m,\triangle_n)$ in $\RR^{n+nm}$:
\begin{align*}
0,\ 2e_i,\ e_i+e_{ik},\ e_i-e_{ik},\ (e_i+e_j)+(e_{ik}-e_{jk})\ &\text{with}\ i\not=j,\\
&i,j=1,\ldots,n,\quad k=1,\ldots,m,
\end{align*}
where $e_\nu$ denotes the standard basis vector in the direction labeled by $\nu$ for every (single or double) index $\nu$.
\end{corollary}

To parse the notation in the corollary, the vertices of $\Hom(\Box_m,\triangle_n)$ are the vertices of a complex of crosspolytopes arranged on the edges of the dilated $n$-simplex, $\conv(0,2e_1,\ldots,2e_n)$, embedded in $\RR^{mn+n}$. For every coordinate direction $e_i$ we get $m$ more coordinate directions, the $e_{ik}$, thus filling out $\RR^{mn+n}$. When viewed as such, the vertices $0$, $2e_i$, $i=1,\ldots,n$, are the vertices of the simplex, the vertices $e_i\pm e_{ik}$ are the vertices of the crosspolytopes centered on the coordinate edges of the simplex, and the vertices $(e_i +e_j)+(e_{ik}-e_{jk})$ are the vertices of the crosspolytopes centered on the non-coordinate edges of the simplex.  Figure 2 represents a geometrically symmetrized version of the above realization of $\Hom(\Box_2,\triangle_3)$.

\begin{figure}[htb]
\includegraphics[scale=.25]{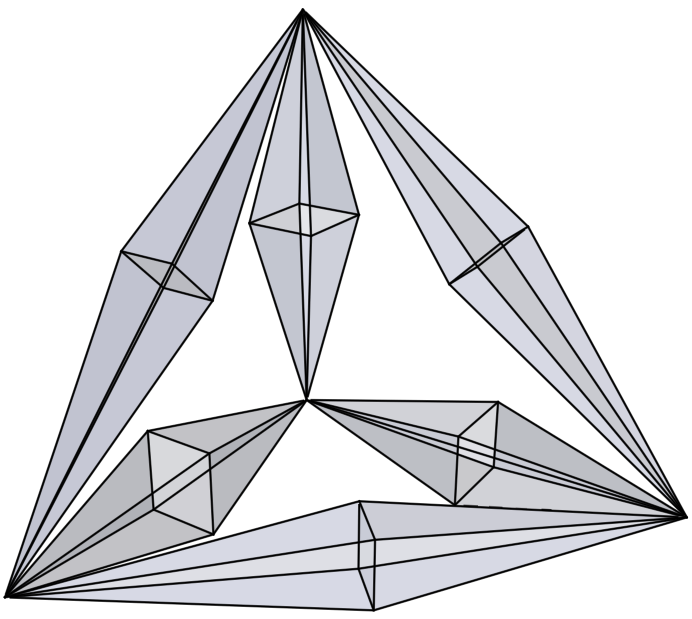}
\caption{}
\end{figure}

\begin{proof}[Proof of Corollary \ref{BT}] By Proposition \ref{calculus}(b), we have the family of $(m+1)$-dimensional crosspolytopes in $\Hom(\Box_m,\triangle_n)$:
$$
\Diamond^{(i)}=\Hom(\Box_m,[0,2e_i])\subset\Hom(\Box_m,\triangle_n),\quad i=1,\ldots,n.
$$
By Proposition \ref{calculus1}(c), we also have
$$
\vertex(\Diamond^{(i)})\subset\vertex(\Box_m,\triangle_n),\qquad i=1,\ldots,n.
$$
By the same Proposition \ref{calculus1}(c), for every facet $F\subset\Box_m$ and every index $i$ we have a pair of rank 1 vertices in $\Diamond^{(i)}$. Depending on whether the function
\begin{align*}
\RR^m\to\RR,\quad (a_1,\ldots,a_n)\mapsto a_k,\quad e_k\perp F,
\end{align*}
is increasing or decreasing, the corresponding elements of the mentioned pair will be denoted by $v_i^+(F)$ and $v_i^-(F)$.

For every pair of indices $i,j\in\{1,\ldots,n\}$, the assignment
$$
v_i^+(F)\mapsto v_j^+(F),\quad v_i^-(F)\mapsto v_j^-(F),\quad F\in\FF(\Box_n),
$$
gives rise to a bijective correspondence between $\vertex(\Diamond^{(i)})$ and $\vertex(\Diamond^{(j)})$.

Let $F\subset\Box_m$ be a facet and $\pi:\Box_m\to[-1,1]$ an $F$-projection. Consider the affine isomorphisms:
$$
\tau_i,\rho_i:[-1,1]\to[0,2e_i],\quad\tau_i(-1)=0,\quad\rho_i(-1)=2e_i,\quad i=1,\ldots,n.\\
$$
The maps
$$
(\tau_i\circ\pi+\rho_j\circ\pi):\Box_m\to[2e_i,2e_j]
$$
are surjective $F$-projections for all $i,j=1,\ldots,n$, $i\not=j$. In terms of the vertices of the $\Diamond^{(i)}$ introduced above, this means
$$
v_i^+(F)+v_j^-(F)\ \in\ \vertex(\Box_m,\triangle_n),
$$
the sum on the left being the point-wise addition of maps evaluating in $\RR_+\triangle_n=\RR^n$.

On the other hand, we have $(e_i+e_j)+(e_{ik}-e_{jk})=(e_i+e_{ik})+(e_j-e_{jk})$ and
$$
\{e_i,e_i+e_{ik},e-e_{ik}\}_{i=1,\ldots,n;\ k=1,\ldots,m}\subset\RR^{mn+n}
$$
is a linearly independent system. So the polytope with vertices as in Corollary \ref{BT} maps to $\Hom(\Box_m,\triangle_n)$ by the affine map that sends $0,2e_1,\ldots,2e_n$ to the corresponding elements of $\vertex^{(0)}(\Box_m,\triangle_n)$ and sends $e_i\pm e_{ik}$ to the corresponding $v_i^{\pm}(F)$. By Theorem \ref{boxtriangle} this is a surjective map. But the rank of the vector system in Corollary \ref{BT} equals $n+nm=\dim(\Hom(\Box_m,\triangle_n))$ (Theorem \ref{folklore}(c)). So the map is an isomorphism.
\end{proof}

In the proof of Theorem \ref{boxtriangle} we will need the following

\begin{lemma}\label{simplexindiamond}
For any natural numbers $2\le n\le m$ and an $n$-simplex $\triangle\subset\Diamond_m$ with $0\in\int(\triangle)$, there is a smooth perturbation $\{g_t\}_{(-1,1)}\subset\Hom(\triangle,\Diamond_m)$ of the identity embedding $g_0:\triangle\to\Diamond_m$, such that $0\in\int(g_t(\triangle))$ for all $t\in(-1,1)$.
\end{lemma}

\begin{proof} Let $\{v_1,\ldots,v_k\}=\vertex(\Diamond_m)\cap\triangle$ for some $k$.

We have $k\ne n+1$, for otherwise either $\triangle\subset\partial\Diamond_m$ or $0$ belongs to an edge of $\triangle$, both possibilities excluded by the condition $0\in\int(\triangle)$.

First consider the case $n=m$. The identity embedding $\triangle\to\Diamond_m$ does not belong to $\vertex(\triangle,\Diamond_m)$ for otherwise Proposition \ref{calculus}(a) implies $\vertex(\triangle)\subset\vertex(\Diamond_m)$, which we have already excluded. Consequently, $\triangle$ can be perturbed inside $\Diamond_m$ smoothly in such a way that $0$ is in the relative interior of the perturbed copies of $\triangle$.

So we can assume $n<m$. Assume, further, $k=n$. Because $0\in\int(\triangle)$ we have $\triangle\subset(\RR v_1+\cdots+\RR v_n)\cap\Diamond_m$. But $(\RR v_1+\cdots+\RR v_n)\cap\Diamond_m\subset\Diamond_m$ is an $n$-dimensional sub-crosspolytope and the problem reduces to the full dimensional case, considered above.

The general case has been reduced to the case $k<n<m$, and the proof is completed by the next general lemma.
\end{proof}

\begin{lemma}\label{newlemma}
Assume $k<n<m$ are three natural numbers, $Q\subset\RR^m$ is an $m$-polytope, $z\in\int(Q)$, and $\triangle\subset Q$ is an $n$-simplex such that $z\in\int(\triangle)$ and $\vertex(\triangle)\cap\vertex(Q)=\{v_1,\ldots,v_k\}$. Then there exists a smooth perturbation
$$
\{g_t\}_{(-1,1)}\subset\Hom(\triangle,Q)
$$
of the identity embedding $g_0:\triangle\to Q$ such that $z\in\int(g_t(\triangle))$ for all $t\in(-1,1)$.
\end{lemma}

\begin{proof}
The set $\vertex(\triangle)\setminus\vertex(Q)$ has at least two elements, say $v$ and $v'$.

Without loss of generality we assume $Q\subset\RR^m$ and $z=0$. Choose a system of hyperplanes $H_1,\ldots,H_{m-n}\subset\RR^m$, satisfying the condition
$$
H_1\cap\cdots\cap H_{m-n}=\RR\triangle.
$$
Consider the polytope $T=\RR\triangle\cap Q$. Without loss of generality,
\begin{equation}\label{vertextriangle}
\vertex(\triangle)\subset\vertex(T),
\end{equation}
for, otherwise, the conditions $\dim\triangle=\dim T$ and $0\in\int(\triangle)$ allow a smooth perturbation $\{g_t\}_{(-1,1)}\subset\Hom(\triangle,T)\subset\Hom(\triangle,Q)$ with $0\in\int(g_t(\triangle))$ for all $t\in(-1,1)$.

The equality $\vertex(\triangle)\cap\vertex(Q)=\{v_1,\ldots,v_k\}$ and (\ref{vertextriangle}) imply that every vertex $w\in\vertex(\triangle)\setminus\{v_1,\ldots,v_k\}$ belongs to a (unique) \emph{positive dimensional} face $F_w\subset Q$, such that
$$
w=\int(F_w)\cap H_1\cap\cdots\cap H_{m-n}.
$$
Denote $d_w=\dim F_w$. We have $d_w\le m-n$ and there exists a $d_w$-element subset $J_w\subset\{1,\ldots,m-n\}$, such that

\begin{equation}\label{vertexhit}
w=\int(F_w)\bigcap\big(\bigcap_{J_w}H_j\big).
\end{equation}
(This follows, for instance, from Lemma \ref{helly}, applied to $w$ instead of $0$ and the system of affine spaces $\Aff(F_w)\cap H_1,\ldots,\Aff(F_w)\cap H_{m-n}$ instead of linear subspaces.)

Next we choose a codimension two subspace $H\subset\RR^m$, such that
$H\cap\RR\triangle=\RR v_1+\cdots+\RR v_k$. In particular, $\vertex(\triangle)\setminus\{v_1,\ldots,v_k\}\cap H=\emptyset$.

For a small open arc $\Gamma\subset S^1$, containing $1$, and any vertex $w\in\vertex(\triangle)\setminus\{v_1,\ldots,v_k\}$, we have
\begin{equation}\label{rotatingintersection}
\dim\big(F_w\bigcap\rho_t^H\big(\bigcap_{J_w}H_j\big)\big)=\dim\big(F_w\bigcap\big(\bigcap_{J_w}\rho_t^H(H_j)\big)\big)=0,\qquad t\in\Gamma,
\end{equation}
where, as in the proof of Lemma \ref{insidesimplex}, $\rho_t^H$ is the corresponding rotation around $H$. The equality (\ref{rotatingintersection}) follows from the equality (\ref{vertexhit}): two complementary dimensional affine subspaces of $\RR^n$ in general position remain in general position after small perturbations.

Since the points
$$
w_t=F_w\bigcap\big(\bigcap_{J_w}\rho_t^H(H_j)\big),\qquad t\in\Gamma,
$$
are smooth non-constant functions of $t$, we get a smooth family in $\Hom(\triangle,\Diamond_m)$:
\begin{align*}
g_t:\triangle\to\Diamond_m,\qquad g_t(w)=
\begin{cases}
w,\ \text{if}\ w\in\{v_1,\ldots,v_k\},\\
w_t,\ \text{if}\ w\in\vertex(\triangle)\setminus\{v_1,\ldots,v_k\},
\end{cases}&\\
&t\in\Gamma.
\end{align*}
Using a diffeomorphism $\Gamma\to(-1,1)$, which maps $1$ to $0$, we obtain a smooth family with the desired property.
\end{proof}

\medskip\begin{proof}[Proof of Theorem  \ref{boxtriangle}]
Pick $f\in\Hom(\Box_m,\triangle_n)$ with $\rank f\ge2$. We want to show that $f$ admits a smooth perturbation inside $\Hom(\Box_m,\triangle_n)$.

By Lemma \ref{insidesimplex}, there is no loss of generality in assuming that $m\ge\rank f=n$. So $f(\Box_m)$ is an $n$-dimensional zonotope in $\triangle_n$.

The inclusion $f(\Box_m)\subset\triangle_n$ gives rise to the inclusion
\begin{equation}\label{prism}
\Box_m\subset\tilde f^{-1}(\triangle_n)=\triangle\times V
\end{equation}
where:

\noindent$\centerdot$ $\tilde f:\RR^m\to\RR^n$ is the unique affine extension of $f:\Box_m\to\triangle_n$,

\noindent$\centerdot$ $V=\tilde f^{-1}(f(0))\subset\RR^m$, an $(m-n)$-dimensional linear subspace,

\noindent$\centerdot$ $\triangle$ is the cross section of $\tilde f^{-1}(\triangle_n)$ by the $n$-dimensional linear subspace of $\RR^m$, \emph{perpendicular} to $V$.

\smallskip Since $f(0)\in\int(f(\Box_m))$, we also have $0\in\int(\triangle)$.

By dualizing, (\ref{prism}) implies
$$
\triangle'\subset\Diamond_m,\qquad 0\in\int(\triangle'),
$$
where $\triangle'$ is the $n$-simplex, dual to $\triangle$ within the subspace $\Aff(\triangle)=\RR\triangle\subset\RR^m$ w.r.t. the Euclidean norm, induced from $\RR^m$.

By Lemma \ref{simplexindiamond}, there is a smooth perturbation $\{g_t\}_{(-1,1)}\subset\Hom(\triangle',\Diamond_m)$ of the identity embedding $g_0:\triangle'\to\Diamond_m$, satisfying the condition $0\in \int(g_t(\triangle'))$ for all $t\in(-1,1)$. For every $t\in(-1,1)$, the dual of $g_t(\triangle')$ in $\RR^m$ is the right prism $\triangle_t\times V_t\subset\RR^m$ for a uniquely determined $n$-simplex $\triangle_t$ with $0\in\int(\triangle_t)$ and the corresponding $(m-n)$-dimensional perpendicular subspace $V_t\subset\RR^m$. In fact, $\triangle_t$ is the dual of $g_t(\triangle')$ within the linear subspace $\RR g_t(\triangle')\subset\RR^m$ w.r.t. the Euclidean norm, induced from $\RR^m$. In particular, $\triangle_0=\triangle$ and $V_0=V$.

By dualizing, the inclusions $g_t(\triangle')\subset\Diamond_m$ imply
\begin{equation}\label{tiltedprism}
\Box_m\subset\triangle_t\times V_t,\qquad t\in(-1,1).
\end{equation}

We can choose two smooth families:

\smallskip\noindent$\centerdot$ $\{\phi_t:\RR\triangle\to\RR\triangle_t\ |\ \phi_t\ \text{a linear isomorphism}\}_{(-1,1)}\subset\Hom(\RR\triangle,\RR^m)\cong\RR^{mn}$,

\smallskip\noindent$\centerdot$ $\{\psi_t:V\to V_t\ |\ \psi_t\ \text{a linear isomorphism}\}_{(-1,1)}\subset\Hom(V,\RR^m)\cong\RR^{m(m-n)}$,

\smallskip\noindent satisfying $\psi_0={\bf1}_V$, $\phi_0={\bf1}_{\triangle}$, and $\phi_t(\triangle)=\triangle_t$ for all $t\in(-1,1)$.

Consider the smooth family of linear automorphisms
\begin{align*}
\alpha_t=\phi_t^{-1}\times\psi_t^{-1}:\RR^m\to\RR^m,\quad(x,y)\mapsto&(\phi_t^{-1}(x),\psi_t^{-1}(y))\\
&(x,y)\in\triangle_t\times V_t,\quad t\in(-1,1).
\end{align*}
We have $\alpha_0={\bf1}_{\RR^m}$.

By (\ref{tiltedprism}), we have the inclusions
$$
\alpha_t(\Box_m)\subset\triangle\times V,\qquad t\in(-1,1),
$$
and, consequently,
$$
(\tilde f\circ\alpha_t)(\Box_m)\subset\triangle_n,\qquad t\in(-1,1).
$$

If the set $(\tilde f\circ\alpha_t)(\Box_m)$ varies along with $t$ then $(\tilde f\circ\alpha_t)|_{\Box_m}$ is a smooth perturbation of $f$ and we are done by Lemma \ref{perturbation}. So
without loss of generality we can assume $(\tilde f\circ\alpha_t)|_{\Box_m}$ does not vary along with $t$. Because $\Box_m$ is full dimensional, this means
\begin{align*}
\alpha_t=
\begin{pmatrix}
\beta_t&0\\
\gamma_t&{\bf 1}
\end{pmatrix},\qquad t\in(-1,1),
\end{align*}
where, for each $t$:

\noindent$\centerdot$ $\beta_t:V\to V$ is a linear automorphism,

\noindent$\centerdot$ $\gamma_t:V\to\RR\triangle$ is linear map,

\noindent$\centerdot$ ${\bf 1}:\RR\triangle\to\RR\triangle$ is the identity map,

\noindent$\centerdot$ $\RR^m$ is thought of as
$\tiny{
\begin{pmatrix}
V\\
\RR\triangle
\end{pmatrix}
}$.

\smallskip Since $\triangle_t$ is the compact perpendicular cross-section of the infinite prism
$$
\alpha^{-1}_t(\triangle\times V)=\triangle_t\times V_t
$$
and
\begin{align*}
\alpha_t^{-1}=
\begin{pmatrix}
\beta_t^{-1}&0\\
\gamma'_t&{\bf 1}
\end{pmatrix}
\end{align*}
for an appropriate $\gamma'_t$, we get $\triangle_t=\triangle$ for all $t\in(-1,1)$, contradicting the definition of the $\triangle_t$.
\end{proof}

\section{$\vertex(\Diamond_m,\triangle_n)$}\label{Diamondtosimplex}

By Proposition \ref{calculus1}(b) and Lemma \ref{insidesimplex}, for determination of the vertices as in the title above, it is enough to consider the rank $k$ vertices of $\Hom(\Diamond_m,\triangle_k)$ for all $k\le n$. The main result of this section is stated in Theorem \ref{diamondtosimplex}, the part (a) of which says that, given an arbitrary $n$-polytope $P$ and a vertex $f\in\vertex^{(n)}(\Diamond_m,P)$, the image $\Im f$ contains an $n$-dimensional crosspolytope, sitting imperturbably in $P$. The parts (b,c) focus on the case when $P=\triangle_n$.  Theorem \ref{diamondtosimplex} allows us to give a geometric description of all vertex maps $\Diamond_m\to\triangle_n$ and estimate their number (Corollary \ref{numberdiamondsimplex}).

We begin by introducing some integer sequences that will help us with the estimation.

Consider the set of vertex maps:
$$
\V(n)=\{f\in\vertex^{(n)}(\triangle_n,\Box_n)\ |\ 0\in\int(\Im f)\}.
$$

We can interpret the elements of $\V(n)$ as the ordered $(n+1)$-tuples of vertices of $\Box_n$, whose convex hulls are full-dimensional and contain $0$ in the interior.

The left action of $\BC_n$ on $\Hom(\triangle_n,\Box_n)$ restricts to a left actions on $\V(n)$.
For the number of orbits of the group action denote
$\beta(n)=\#\big(\BC_n\backslash\V(n)\big)$.

\begin{lemma}\label{diamondembed}
Let $n$ be a natural number.
\begin{itemize}
\item[(a)] $\#\V(n)=2^nn!\beta(n)$,
\item[(b)]
$\beta(n)=
\begin{cases}
1,\ \text{if}\ n=1,3,\\
0,\ \text{if}\ n=2,\\
\ge1,\ \text{if}\ n\ge4.
\end{cases}$
\item[(c)] Assuming $n\ge 2$, there is a bijection $\vertex^{(n)}(\Diamond_n,\triangle_n)\approx\V(n)$.
\end{itemize}
\end{lemma}

\begin{proof} (a) follows from the fact that $\BC_n$ acts on $\V(n)$ freely.

\smallskip\noindent(b) The equality is obvious for $n=1,2$. For $n=3$, we first observe that $\V(n)/S_{n+1}$ has two elements, represented by the two maximal regular tetrahedra $\triangle,\triangle'\subset\Box_3$. Moreover, $\triangle$ can be mapped to $\triangle'$ by the $90^{\text o}$-rotation around the line through the centers of a pair of opposite facets of $\Box_3$. So it is enough to show that every automorphism of $\triangle$ extends to an automorphism of $\Box_3$. To this end, pick any two vertices $x,y\in\vertex(\triangle)$. The reflection of $\RR^3$ w.r.t. the affine plane, perpendicular to $[x,y]$ and through $\frac{x+y}2$, swaps $x$ and $y$ and leaves the other two vertices of $\triangle$ fixed. We are done because this reflection is an element of $\BC_3$ and transpositions generate $\Aut(\triangle)\cong S_4$.

For the inequalities we exhibit the following explicit element of $\V(n)$:
$$
\conv\big((-1,\ldots,-1),\big\{(1,\ldots,1)-2e_i\big\}_{i=1}^n\big)\subset\Box_n.
$$

\smallskip\noindent(c) Pick $f\in\vertex^{(n)}(\Diamond_n,\triangle_n)$. By Lemma \ref{insidesimplex}, every facet of $\triangle_n$ contains a facet of $f(\Diamond_n)$. By applying an affine isomorphism $\RR^n\to\RR^n$, transforming $f(\Diamond_n)$ into $\Diamond_n$, the simplex $\triangle_n$ gets transformed into an $n$-simplex, such that the condition on the facets is still satisfied. Dualizing, we get an element of $\V(n)$. This is a bijective correspondence.
\end{proof}

\smallskip\noindent\emph{Notice.} By computer assisted effective methods, we have computed the following values: $\beta(4)=5$ and $\beta(5)=408$.

When $n\ge4$, the vertex map $f:\Diamond_n\to\triangle_n$, corresponding to the explicit element of $\V(n)$ in the proof of Lemma \ref{diamondembed}(c), does not map $0$ to the barycenter of $\triangle_n$. To see this, we change $\triangle_n$ to the regular $n$-simplex $\triangle=\conv(e_1,\ldots,e_{n+1})\subset\RR^{n+1}$ and look at the corresponding vertex map $g:\Diamond_n\to\triangle$. Let $\gamma=\frac1{n+1}\sum_{i=1}^{n+1}e_i$. We want to show $g(0)\not=\gamma$. The dual to $\triangle$ w.r.t. $\gamma$ is a homothetic image $\triangle'$ of $\triangle$, centered at $\gamma$. If $g(0)=\gamma$ then $\triangle'$ sits in an $n$-parallelepiped the same way as $\conv\big((-1,\ldots,-1),\big\{(1,\ldots,1)-2e_i\big\}_{i=1}^n\big)$ sits in $\Box_n$, with $\gamma$ playing the same role in $\triangle'$ as $0$ in $\conv\big((-1,\ldots,-1),\big\{(1,\ldots,1)-2e_i\big\}_{i=1}^n\big)$. But this contradicts the equality $\sum_{\vertex(\triangle')}(v-\gamma)=0$.

It is interesting to remark that any simple $3$-dimensional polytope $P$ contains a homothetic copy $\Diamond$ of the octahedron $\Diamond_3$, such that $\vertex(\Diamond)\subset\partial P$ \cite{Akopyan}.

\medskip For two natural numbers $m$ and $n$, denote by $\Sigma(m,n)$ the set of maps
$$
f:\{1,\ldots,m\}\to\{\pm1,\ldots,\pm n\},
$$
such that
$$
\{|f(1)|,\ldots,|f(m)|\}=\{1,\ldots,n\}.
$$
Let $\sigma(m,n)=\#\Sigma(m,n)$. (So $\sigma(m,n)=0$ for $m<n$.)

\begin{lemma}\label{stirlingnumebrs} For all $m\ge n$ we have
$$
\sigma(m,n)=2^mT(m,k)=2^mn!S(m,n)=2^m\sum_{j=0}^n(-1)^{n-j}{n\choose j}j^m,
$$
where $T(m,n)$ is the number of surjective maps $\{1,\ldots,m\}\to\{1,\ldots,n\}$ and $S(m,n)$ is the Stirling number of the second kind, i.e., the number of partitions of $m$ objects into $n$ non-empty subsets.
\end{lemma}

The numbers $T(m,n)$ and $S(m,n)$ are listed in \cite{Sloane} as, correspondingly, the sequences A019538 and A008277.

\begin{lemma}\label{symmetricintersection}
For any two $n$-dimensional simplices $\triangle,\triangle'\subset\RR^n$ we have
$$
\#\vertex(\triangle\cap\triangle')\le{2n+2\choose n+2}.
$$
\end{lemma}

The following table explains why one should expect a substantial improvement in the upper bound for $\#\vertex(\triangle\cap\triangle')$ in the relevant case when $\triangle$ and $\triangle'$ are mutually centrally symmetric; see the notice after the proof of Corollary \ref{numberdiamondsimplex}. The five entries row-wise are, correspondingly, the dimension $n$, $\#\vertex\big(\triangle_n\cap(2z-\triangle_n)\big)$ for a small random perturbation $z$ of the barycenter of $\triangle_n$ (conjecturally maximizing the number of vertices), $\#\vertex(\triangle\cap\triangle')$ for some randomly generated $n$-simplices $\triangle,\triangle'\subset\RR^n$ (up to dimension 13),  ${2n+2\choose n+2}$, and the percentage of the second number to the fourth:
$$
\begin{matrix}
3&12&12&56&21.43\ \%\\
4&30&31&210&14.29\ \%\\
5&60&64&792&7.576\ \%\\
6&140&144&3003&4.662\ \%\\
7&280&334&11440&2.448\ \%\\
8&630&781&43758&1.440\ \%\\
9&1260&1586&167960&0.7502\ \%\\
10&2772&3623&646646&0.4287\ \%\\
11&5544&8912&2496144&0.2221\ \%\\
12&12012&18155&9657700&0.1244\ \%\\
13&24024&43678&37442160&0.06416\ \%\\
14&51480&-&145422675&0.03540\ \%\\
15&102960&-&565722720&0.01820\ \%\\
\end{matrix}
$$

\begin{proof}[Proof of Lemma \ref{symmetricintersection}]
Assume $\{v_1,\ldots,v_m\}=\vertex(\triangle\cap\triangle')$ for some $m\in\NN$.

For every index $i\in\{1,\ldots,m\}$ we let $F_i\subset\triangle$ and $G_i\subset\triangle'$ be the faces, uniquely determined by the condition $v_i=\int(F_i)\cap\int(G_i)$.

For every index $i\in\{1,\ldots,m\}$ there are faces $\tilde F_i\subset\triangle$ and $\tilde G_i\subset\triangle'$, such that $F_i\subset\tilde F_i$, $G_i\subset\tilde G_i$, $v_i=\tilde F_i\cap\tilde G_i$, and $\dim\tilde F_i+\dim\tilde G_i=n$. In fact, if $F_i=A_1\cap\cdots\cap A_k$ and $G_i=B_1\cap\cdots\cap B_l$ for the corresponding facets $A_1,\ldots,A_k\in\FF(\triangle)$ and $B_1,\ldots,B_l\in\FF(\triangle')$, where $k=\codim F_i$ and $l=\codim G_i$, then Lemma \ref{helly} implies $C_1\cap\cdots\cap C_n=v_i$ for some $n$-element subset
$$
\{C_1,\ldots,C_n\}\subset\{A_1,\ldots,A_k,B_1,\ldots,B_l\}.
$$
So we can choose
$$
\tilde F_i=\bigcap_{C_j\in\{A_1,\ldots,A_k\}} C_j,\qquad \tilde G_i=\bigcap_{C_j\in\{B_1,\ldots,B_l\}}C_j.
$$

The existence of the $\tilde F_i$ and $\tilde G_i$ implies
$$
m\le\sum_{p=0}^n{n+1\choose p+1}{n+1\choose n+1-p}=\sum_{q=1}^{n+1}{n+1\choose q}{n+1\choose n+2-q}={2n+2\choose n+2},
$$
the last equality resulting from the Chu-Vandermonde identity \cite[Chapter 1]{RioCOMB}.
\end{proof}

Our main result in this section is the following

\begin{theorem}\label{diamondtosimplex}
Let $m\ge n\ge 2$ be natural numbers and $P$ an $n$-polytope.
\begin{itemize}
\item[(a)]
Every element $f\in\vertex^{(n)}(\Diamond_m,P)$ fits in a commutative diagram
\[
\begin{tikzpicture}[description/.style={fill=white,inner sep=2pt}]
\matrix (m) [matrix of math nodes, row sep=3em,
column sep=2.5em, text height=1.5ex, text depth=0.25ex]
{ \Diamond_m & & P \\
& \Diamond_n & \\ };
\path[->,font=\scriptsize]
(m-1-1) edge node[auto] {$ f $} (m-1-3)
(m-2-2) edge node[auto, swap] {$\nu_2$} (m-1-3)
edge node[auto] {$\nu_1$} (m-1-1);
\end{tikzpicture}
\]
where $\nu_1$ and $\nu_2$ are injective vertex maps and $\nu_1(0)=0$.
\item[(b)] $\#\{f\in\vertex^{(n)}(\Diamond_m,\triangle_n)\ |\ \Im f\cong\Diamond_n\}=\sigma(m,n)\beta(n)$.
\item[(c)] All elements of $\vertex^{(n)}(\Diamond_m,\triangle_n)$ satisfy the condition $\Im f\cong\Diamond_n$ if and only if either $m=n$ or $n=3$. (Figure 3 represents the claim (c) for $n=3$.)
\begin{figure}[htb]
\includegraphics[trim=50mm 180mm 50mm 40mm, clip]{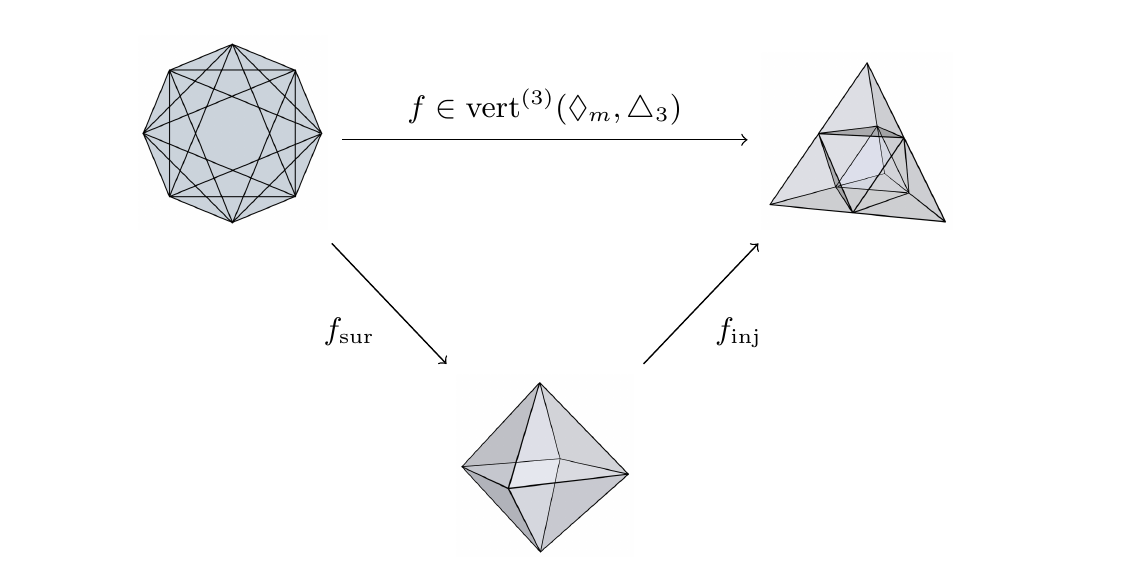}
\caption{}
\end{figure}
\end{itemize}
\end{theorem}

\begin{corollary}\label{numberdiamondsimplex}
Let $m$ and $n$ be natural numbers, $m\ge2$, and $P$ be an $n$-polytope.
\begin{itemize}
\item[(a)] If $m\ge n$ then the elements $f\in\vertex^{(n)}(\Diamond_m,P)$ are exactly the maps $f:\Diamond_m\to P$ admitting a subset $\{i_1,\ldots,i_n\}\subset\{1,\ldots,m\}$ such that $f|_{\Diamond}\in\vertex^{(n)}(\Diamond,P)$ and $\{f(e_j),f(-e_j)\}=\{v_j,2f(0)-v_j\}$, where $\Diamond=\conv(\pm e_{i_1},\ldots,\pm e_{i_n})$ and $v_j\in\vertex(P\cap(2f(0)-P))$ for $j\in\{1,\ldots,n\}\setminus\{i_1,\ldots,i_n\}$.
\item[(b)]
\begin{align*}
\#\vertex(\Diamond_m,\triangle_n)=1+n+&2^{m-1}n(n+1)+{n+1\choose4}\sigma(m,3)+\\
&+\sum_{k=4}^{\min(m,n)}{n+1\choose k+1}\cdot\#\vertex^{(k)}(\Diamond_m,\triangle_k).
\end{align*}
\item[(c)] For all $k\le m$ we have
$$
\sigma(m,k)\beta(k)\le\#\vertex^{(k)}(\Diamond_m,\triangle_k)\le\frac{2^km!}{(m-k)!}{2k+2\choose k+2}^{m-k}\beta(k).
$$
\end{itemize}
\end{corollary}

\begin{proof} (a) That the mentioned maps $f:\Diamond_m\to P$ are vertex maps follows from the perturbation criterion, and that no element of $\vertex^{(n)}(\Diamond_m,P)$ is left out is the contents of Theorem \ref{diamondtosimplex}(a), with help from Proposition \ref{calculus1}(a) and Corollary \ref{convexposition1}(b).

\medskip\noindent(b) Lemma \ref{insidesimplex} implies
$$
\#\vertex(\Diamond_m,\triangle_n)=\sum_{k=0}^{\min(m,n)}{n+1\choose k+1}\cdot\#\vertex^{(k)}(\Diamond_m,\triangle_k).
$$
So the equality for $\#\vertex(\Diamond_m,\triangle_n)$ follows because:
$$
\#\vertex^{(k)}(\Diamond_m,\triangle_k)=
\begin{cases}
1\ \text{if}\ k=0,\\
2^m\ \text{if}\ k=1,\ \text{by Proposition \ref{calculus1}(c)},\\
0\ \text{if}\ k=2,\ \text{by Lemma \ref{diamondembed} and Theorem \ref{diamondtosimplex}(a)},\\
\sigma(m,3)\ \text{if}\ k=3\le m,\ \text{by Theorem \ref{diamondtosimplex}(b,c)}\ \text{and}\ \beta(3)=1.
\end{cases}
$$

\medskip\noindent(c) The lower bound follows from Theorem \ref{diamondtosimplex}(b).

By Lemma \ref{diamondembed}, there are $2^kk!\beta(k)$ possibilities for $\nu_2$ in the diagram in Theorem \ref{diamondtosimplex}(a), with $P=\triangle_k$.

Proposition \ref{calculus1}(a) and Corollary \ref{convexposition1}(b) give rise to a bijective correspondence between the elements $\nu\in\vertex^{(k)}(\Diamond_k,\Diamond_m)$ with $\nu(0)=0$ and the set of injective maps $\{\pm e_1,\ldots,\pm e_k\}\to\{\pm e_1,\ldots,\pm e_m\}$, mapping antipodes to antipodes.
So there are $2^k\frac{m!}{(m-k)!}$ possibilities for $\nu_1$ in the same diagram.

By (a), for fixed $\nu_1$ and $\nu_2$, the set of $f\in\vertex^{(k)}(\Diamond_m,\triangle_k)$, fitting in the diagram in Theorem \ref{diamondtosimplex}(a) with $P=\triangle_k$, is bijective to the set of maps of type
\begin{align*}
\psi:\{\pm e_i\ |\ 1\le i\le k,\ i\not=i_1,\ldots,i_k\}\to\vertex(&\triangle_k\cap(2\phi(0)-\triangle_k)),\\
&\psi(-e_i)=2\nu_2(0)-\psi(e_i).
\end{align*}
There are $\#\vertex(\triangle_k\cap(2\phi(0)-\triangle_k))^{m-k}$ possibilities for such $\psi$. By Lemma \ref{symmetricintersection}, the number of the $f$ for fixed $\nu_1$ and $\nu_2$ is bounded above by ${2k+2\choose k+2}^{m-k}$.

Next we reduce the multiplicities in our counting as follows: the set of the vertex maps $f:\Diamond_m\to\triangle_k$ for a pair $(\nu_1,\nu_2)$ equals the set of those for $(\nu_1\alpha,\alpha^{-1}\nu_2)$ for any $\alpha\in\BC_k$. The $\BC_k$-action on the pairs $(\nu_1,\nu_2)$, given by $\alpha*(\nu_1,\nu_2)=(\nu_1\alpha,\alpha^{-1}\nu_2)$, is free. So the product of the numbers of possibilities for $\nu_1$ and $\nu_2$ and the upper bound ${2k+2\choose k+2}^{m-k}$, divided by $|\BC_k|=2^kk!$, bounds above $\#\vertex^{(k)}(\Diamond_m,\triangle_k)$.
\end{proof}

\medskip\noindent\emph{Notice}. The inequality in Corollary \ref{numberdiamondsimplex}(c) is sharp in the following sense: when $k=m$ the upper and lower bounds are equal to $2^mm!\beta(m)$. This means that any improvement in the upper bound should come from an improvement in the upper bound in Lemma \ref{symmetricintersection} in the special case when the two simplices are mutually symmetric w.r.t. a point.

\begin{proof}[Proof of Theorem \ref{diamondtosimplex}(a)] We will use the following notation. For any subset $J\subset\{1,\ldots,m\}$ we have the sub-crosspolytope
$$
\Diamond(J)=\conv(\{\pm e_j\ |\ j\in J\}\subset\Diamond_m.
$$

In view of Proposition \ref{calculus1}(a) and Corollary \ref{convexposition1}(b), Theorem \ref{diamondtosimplex}(a) admits the following equivalent reformulation: there exists a subset
$$
\{i_1,\ldots,i_n\}\subset\{1,\ldots,m\}
$$
such that $\Diamond(i_1,\ldots,i_n)$ satisfies the condition:
$$
f|_{\Diamond(i_1,\ldots,i_n)}\in\vertex^{(n)}(\Diamond(i_1,\ldots,i_n),P).
$$

Assume, to the contrary, that such a subset does not exist.

Pick any subset $\{i_1,\ldots,i_n\}\subset\{1,\ldots,m\}$. By Lemma \ref{perturbation}, there is an affine 1-family $\{g_t\}_{(-1,1)}\subset\Hom(\Diamond(i_1,\ldots,i_n),P)$ with $g_0=f|_{\Diamond(i_1,\ldots,i_n)}$.

First we observe that $\{g_t\}$ is not constant on $0$ as $t$ runs over $(-1,1)$. In fact, if $g_t(0)=f(0)$ for all $t$ then the system $\{g_t\}$ gives rise to an affine perturbation of $f$ by extending $g_t$ to the maps $f_t:\Diamond_m\to P$, defined by $f_t(\pm e_i)=f(\pm e_i)$ for all $i\in\{1,\ldots,m\}\setminus\{i_1,\ldots,i_n\}$. But $f$, being a vertex map, can not be perturbed.

By Corollary \ref{convexposition1}(c), we have $f(0)\not=f(\pm e_i)$ for every index $i\in\{1,\ldots,m\}$. We can assume $g_t(0)\not=g_t(\pm e_i)$ for every $t\in(-1,1)$ and every index $i\in\{i_1,\ldots,i_n\}$. Since $\{g_t\}$ is not constant on $0$,  for every index $i\in\{i_1,\ldots,i_n\}$ at least one of the two subsets
$$
\{g_t(e_i)\}_{(-1,1)},\ \{g_t(-e_i)\}_{(-1,1)}\subset P
$$
is an open interval. Pick one such interval per index $i\in\{i_1,\ldots,i_n\}$. We obtain a system of open intervals
$$
I_{i_1},\ldots,I_{i_n}\subset P,
$$
such that either $g_t(e_{i_k})\in I_{i_k}$ for all $t\in(-1,1)$ or $g_t(-e_{i_k})\in I_{i_k}$  for all $t\in(-1,1)$, where $k\in\{1,\ldots,n\}$. For simplicity of notation and without loss of generality, we can assume
$$
g_t(e_{i_k})\in I_{i_k},\quad t\in(-1,1),\quad k=1,\ldots,n.
$$
These points are non-constant affine functions of the parameter $t$.

By Lemma \ref{convexposition1}(c), $g_0(e_{i_k})=f(e_{i_k})\in\vertex\big(P\cap(2f(0)-P)\big)$. Therefore, the unique faces
$$
F_{i_k}, G_{i_k}\subset P,\qquad k=1,\ldots,n,
$$
such that $g_0(e_{i_k})=\int(F_{i_k})\cap\int(2f(0)-G_{i_k})$, satisfy the additional condition
\begin{equation}\label{complementaryspaces}
\lin(F_{i_k})\cap\lin(G_{i_k})=0.
\end{equation}
We also have $I_{i_k}\subset\int(F_{i_k})$ because $I_{i_k}$ is an open interval inside $P$, containing an interior point of the face $F_{i_k}\subset P$.

Similarly, we write
\begin{align*}
&g_0(-e_{i_k})=f(-e_{i_k})\in\int(2f(0)-F_{i_k})\cap\int(G_{i_k}),\\
&\{g_t(-e_{i_k})\}_{(-1,1)}\subset\int(G_{i_k}).
\end{align*}
(Unlike $I_{i_k}$, the set $\{g_t(-e_{i_k})\}_{(-1,1)})$ may be just a single point.) Consequently, as $t$ varies over $(-1,1)$, the point
$g_t(0)=1/2\big(g_t(e_{i_k})+g_t(-e_{i_k})\big)$ traces out an open interval, parallel to $\lin(F_{i_k})+\lin(G_{i_k})$:
$$
\lin\big(\{g_t(0)\}_{(-1,1)}\big)\subset\lin(F_{i_k})+\lin(G_{i_k}),\qquad k=1,\ldots,n.
$$
In particular,
$$
\dim\big(\bigcap_{k=1}^n V_{i_k}\big)>0,\quad\text{where}\ V_{i_k}=\lin F_{i_k}+\lin G_{i_k}\subset\RR^n.
$$

It is crucial that the linear space $V_{i_k}$ depends only on the index $i_k$ and not on the 1-family $\{g_t\}$. So we can introduce the analogous subspaces $V_i\subset\RR^n$ for all $i=1,\ldots,m$. Let $F_i,G_i\subset P$ be the pairs of faces as above, corresponding to the indices $i=1,\ldots,m$. (So we have $\lin F_i\cap\lin(G_i)=0$.)

Since the subset $\{i_1,\ldots,i_n\}\subset\{1,\ldots,m\}$ in the discussion above was arbitrary, Lemma \ref{helly} implies $\bigcap_{i=1}^mV_i\not=0$. Now we can choose a sufficiently small open interval $I$ as follows
$$
f(0)\in I\subset \big(f(0)+\bigcap_{i=1}^mV_i\big)\cap P.
$$
Then (\ref{complementaryspaces}) implies that $F_i\cap(2\tau-G_i)$ and $(2\tau-F_i)\cap G_i$ are singletons for every index $i\in\{1,\ldots,n\}$ and the centrally symmetric subpolytopes
$$
Q_\tau=\conv\big(\{F_i\cap(2\tau-G_i),(2\tau-F_i)\cap G_i\}_{i=1}^n\big)\subset P\cap(2\tau-P)
$$
define a \emph{non-constant} affine deformation of $\Im f$ as $\tau$ varies over $I$; i.e., $Q_0=\Im f$, the vertices of $Q_\tau$ are affine functions of $\tau$, and $\vertex(Q_\tau)$ is bijective to $\vertex(\Im f)$ for every $\tau$. In other words, $f$ admits an affine perturbation -- the desired contradiction.
\end{proof}

\begin{proof}[Proof of Theorem \ref{diamondtosimplex}(b)]
In view Proposition \ref{calculus1}(b), for any rank $n$ vertex map $f:\Diamond_m\to\triangle_n$ with $\Im f\cong\Diamond_n$, the map $f_{\surj}$ can be viewed as a surjective \emph{vertex} map $f_{\surj}:\Diamond_m\to\Diamond_n$. The surjectivity implies $f_{\surj}(0)=0$. On the other hand, Corollary \ref{convexposition1}(b) implies that the set of surjective linear vertex maps $\Diamond_m\to\Diamond_n$ is bijective to $\Sigma(m,n)$. So there are $\sigma(m,n)$ possibilities for $f_{\surj}$.

By Proposition \ref{calculus1}(b), the set of possible $f_{\inj}$ can be identified with $\vertex^{(n)}(\Diamond_n,\triangle_n)$. By composing the linear surjective vertex maps $\Diamond_m\to\Diamond_n$ with the injective vertex maps $\Diamond_n\to\triangle_n$, we can form $\sigma(m,n)\cdot\#\big(\vertex^{(n)}(\Diamond_n,\triangle_n)/\BC_n\big)$ different maps $\Diamond_m\to\triangle_n$. The latter number equals $\sigma(m,n)\beta(n)$ by Lemma \ref{diamondembed}(a,c).

It only remains to show that a rank $n$ map $f:\Diamond_m\to\triangle_n$ is a vertex map if $f_{\surj}:\Diamond_m\to\Diamond_n$ and $f_{\inj}:\Diamond_n\to\triangle_n$ are vertex maps. But if $f_{\surj}$ is a vertex map then, by Corollary \ref{convexposition1}(b), $f(\vertex(\Diamond_m))=\vertex(\Diamond_n)$. On the other hand, if $f$ admits a smooth perturbation $\{f_t\}_{(-1,1)}$, then the latter is not constant on at least one pair of opposite vertices of $\vertex(\Diamond_m)$. But the two conditions together make $f_{\inj}$ perturbable.
\end{proof}

\begin{proof}[Proof of Theorem \ref{diamondtosimplex}(c)]
First we show that, for any natural numbers $m>n>3$, the set $\vertex^{(n)}(\Diamond_m,\triangle_n)$ contains an element whose image is not isomorphic to $\Diamond_n$.

Pick an element $f\in\vertex^{(n)}(\Diamond_n,\triangle_n)$ -- a non-empty set by Theorem \ref{diamondtosimplex}(b). By Corollary \ref{convexposition1}(c),
$$
f(\vertex(\Diamond_n))\subset\vertex(\triangle_n\cap(2f(0)-\triangle_n)).
$$
We have $\#\FF(\triangle_n\cap(2f(0)-\triangle_n))\le2\cdot\#\FF(\triangle_n)=2(n+1)$. On the other hand, the inequality $n>3$ implies $2^n>2(n+1)$. Since $\Im f\cong\Diamond_n$, we have $\#\FF(\Im f)=2^n$. So $\Im f\subsetneq \triangle_n\cap(2f(0)-\triangle_n)$ and we can find two opposite vertices $v,v'\in\vertex(\triangle_n\cap(2f(0)-\triangle_n))\setminus\Im f$. Consider the map
\begin{align*}
f':\Diamond_m\to\triangle_n,\qquad&f'|_{\Diamond_n}=f,\\
&f'(e_i)=v,\ f'(-e_i)=v',\quad i=n+1,\ldots,m.
\end{align*}

Since $\vertex(\Im f)\subsetneq\vertex(\Im\phi)$, the polytope $\Im f'$ is not isomorphic to $\Diamond_n$. We are done, because $f'\in\vertex^{(n)}(\Diamond_m,\triangle_n)$ by Corollary \ref{numberdiamondsimplex}(a). (The latter does not use Theorem \ref{diamondtosimplex}(c).)

\medskip The elements of $\vertex^{(n)}(\Diamond_n,\triangle_n)$ have images isomorphic to $\Diamond_n$. So we can assume $m>n$. By Theorem \ref{diamondtosimplex}(b), $\vertex^{(2)}(\Diamond_m,\triangle_2)=\emptyset$. So it only remains to show that $\Im f\cong\Diamond_3$ for any element $f\in\vertex^{(3)}(\Diamond_m,\triangle_3)$.

By Theorem \ref{diamondtosimplex}(a), there exists $g\in\vertex^{(3)}(\Diamond_3,\triangle_3)$ such that $\Im g\subseteq\Im f$. On the one hand, $\Im g=\triangle_3\cap(2g(0)-\triangle_3)$ (Lemma \ref{diamondembed}). On the other hand, $\Im f\subseteq\triangle_3\cap(2g(0)-\triangle_3)$ (Corollary \ref{convexposition1}(c)). Hence $\Im g=\Im f$.
\end{proof}

\section{$\vertex(\Diamond_m,\Diamond_n)$}\label{Diamondtodiamond}

For any $f\in\vertex(\Diamond_m,\Diamond_n)$, whose image is contained in the boundary of $\Diamond_n$,
we defer to Section \ref{Diamondtosimplex}. So the only new situation is when $\Im f\cap\int(\Diamond_n)\not=\emptyset$,
or equivalently, when $f(0)\in\int(\Diamond_n)$. The main result of this section is that in this case $f$ is necessarily a linear map.

\begin{theorem}\label{diamondtodiamond}
For $m,n\ge2$, we have the implication
$$
f\in\vertex(\Diamond_m,\Diamond_n)\ \ \&\ \ f(0)\in\int(\Diamond_n)\quad\Longrightarrow\quad f(0)=0.
$$
\end{theorem}

By Theorem \ref{diamondtosimplex}(b) and Corollary \ref{numberdiamondsimplex}(c), the next corollary yields nontrivial estimates for the number of vertex maps between two crosspolytopes of arbitrary dimension.

\begin{corollary}\label{numberdiamonddiamond}
For $m,n\ge2$ we have
\begin{align*}
\#\vertex(\Diamond_m,\Diamond_n)=2^mn^m&+2n+2^{m+1}n(n-1)+16{n\choose4}\sigma(m,3)+\\
&+\sum_{k=4}^{\min(m,n-1)}2^{k+1}{n\choose k+1}\cdot\#\vertex^{(k)}(\Diamond_m,\triangle_k).
\end{align*}
\end{corollary}

\begin{proof}
By Corollary \ref{convexposition1}(b) and Theorem \ref{diamondtodiamond}, the elements $f\in\vertex(\Diamond_m,\Diamond_n)$ with $f(0)\in\int(\Diamond_n)$ are in bijective correspondence with the maps $\vertex(\Diamond_m)\to\vertex(\Diamond_n)$, mapping antipodes to antipodes. The number of such maps is $2^mn^m$.

By Lemma \ref{insidesimplex}, the number of the vertex maps $f:\Diamond_m\to\Diamond_n$, such that $f(0)$ is in the interior of a $k$-dimensional face of $\Diamond_n$ for some $k<n$, equals $\#\vertex^{(k)}(\Diamond_m,\triangle_k)$ if $m\ge k$ and $0$ otherwise. Since the number of $k$-faces of $\Diamond_n$ is $2^{k+1}{n\choose{k+1}}$ (\cite[Section 7.2]{Coxeter}), we can write
\begin{align*}
\#\vertex(\Diamond_m&,\Diamond_n)=2^mn^m+\sum_{k=0}^{\min(m,n-1)}2^{k+1}{n\choose k+1}\cdot\#\vertex^{(k)}(\Diamond_m,\triangle_k).
\end{align*}

Now the desired expression is obtained from the sum above by explicating the $k$-th summands for $k=0,1,2,3$, just like we did in Corollary \ref{numberdiamondsimplex}(b).
\end{proof}

The rest of the section is a series of lemmas that comprise the proof of Theorem \ref{diamondtodiamond}.

For the coordinate hyperplanes in $\RR^n$ we use the notation
$$
H^{(i)}=\sum_{j\not=i}\RR e_i\subset\RR^n,\qquad\qquad i=1,\ldots,n.
$$

\begin{lemma}\label{vertexunion}
For a point $z\in\RR^n$, satisfying $\dim(\Diamond_n\cap(z+\Diamond_n))=n$, we have
\begin{align*}
\big(\Diamond_n\cap(z+\Diamond_n)\big)&\setminus\big(\int(\Diamond_n)\cup\int(z+\Diamond_n)\big)\subset\\
&\big(\Diamond_n\cap(z+\Diamond_n)\big)\setminus\bigcap_{i=1}^n\big(H^{(i)}\cup\big(z+H^{(i)}\big)\big)\\
\end{align*}
\end{lemma}

\begin{proof}
For simplicity of notation, denote $\Diamond=\Diamond_n\cap(z+\Diamond_n)$.

Consider a subset $J\subset\{1,\ldots,n\}$, such that the point
$$
v_J=\big(\bigcap_J H^{(j)}\big)\ \bigcap\ \big(\bigcap_{\{1,\ldots,n\}\setminus J}\big(z+H^{(k)}\big)\big)\in\RR^n
$$
does not belong to $\int(\Diamond_n)\cup\int(z+\Diamond_n)$. We want to show $v_J\notin\Diamond$.

Let $K=\{1,\ldots,n\}\setminus J$. Observe, $J\not=\emptyset$ and $K\not=\emptyset$ for, otherwise, either $v_J=0\in\int(\Diamond_n)$ or $v_J=z\in\int(z+\Diamond_n)$.

Put $\Diamond(J)=\conv\big(\{\pm e_j\}_J\big)$ and similarly for $\Diamond(K)$ (as in the proof of Theorem \ref{diamondtosimplex}(a)). Assume to the contrary $v_J\in\Diamond$. Then
\begin{align*}
&v_J\in\Diamond_n\ \bigcap\ \big(\bigcap_JH^{(j)}\big)=\Diamond(J),\\
&v_J\in(z+\Diamond_n)\ \bigcap\ \big(\bigcap_K\big(z+H^{(k)}\big)\big)=z+\Diamond(K).
\end{align*}
Since $\int(\Diamond(J)),\int(\Diamond(K))\subset\int(\Diamond_n)$, we conclude
$$
v_J\in\partial\Diamond(J)\ \cap\ \partial\big(z+\Diamond(K))\big)=\partial\Diamond(J)\cap\big(z+\partial\Diamond(K)\big).
$$

Assume
$$
v_J=\conv\big(\{\delta_je_j\}_J\big)\cap\big(z+\conv\big(\{\delta_ke_k\}_K\big)\big)
$$
for some $\delta_1,\ldots,\delta_n\in\{-1,1\}$. Then we have
\begin{align*}
\Aff\big(\{\delta_je_j\}_J,\{-\delta_ke_k\}_K\big)=z+\Aff\big(\{-\delta_je_j\}_J,\{\delta_ke_k\}_K\big),
\end{align*}
i.e., the parallel translation by $z$ moves one facet hyperplane of $\Diamond_n$ to its opposite. But this contradicts the assumption $\dim\Diamond=n$.
\end{proof}

\begin{lemma}\label{cones}
\emph{(a)} Let $C$ and $D$ be cones in $\RR^n$ and $x_1,x_2,x_3$ be three distinct collinear points in $\RR^n$. Assume $x_3\notin[x_1,x_2]$. Then the polyhedra $(x_1+C)\cap(x_3+D)$ and $(x_2+C)\cap(x_3+D)$ are homothetic.

\medskip\noindent\emph{(b)} Let $n\ge 2$ and $z\in\RR^n$. Assume $\dim(\Diamond_n\cap(z+\Diamond_n))=n$. Then there exists a real number $\epsilon>0$ such that for all $\lambda\in(1-\epsilon,1+\epsilon)$ we have the equality of normal fans:
\begin{align*}
\cN\big(\big(e_n+&\RR_+(\Diamond_n-e_n)\big)\bigcap\big(\lambda z-e_n+\RR_+(\Diamond_n+e_n)\big)\big)=\\
&\cN\big(\big(e_n+\RR_+(\Diamond_n-e_n)\big)\bigcap\big(z-e_n+\RR_+(\Diamond_n+e_n)\big)\big).
\end{align*}
\end{lemma}

\bigskip Figure 3 represents the cones $e_n+\RR_+(\Diamond_n-e_n)$ and $\lambda z-e_n+\RR_+(\Diamond_n+e_n)$ and their intersection. Small perturbations $\lambda z$ of $z$ preserve the combinatorial type of the intersection, and thus keep the corresponding normal fans a constant.

\begin{figure}[htb]
\includegraphics[scale=1, trim=40mm 155mm 40mm 60mm, clip]{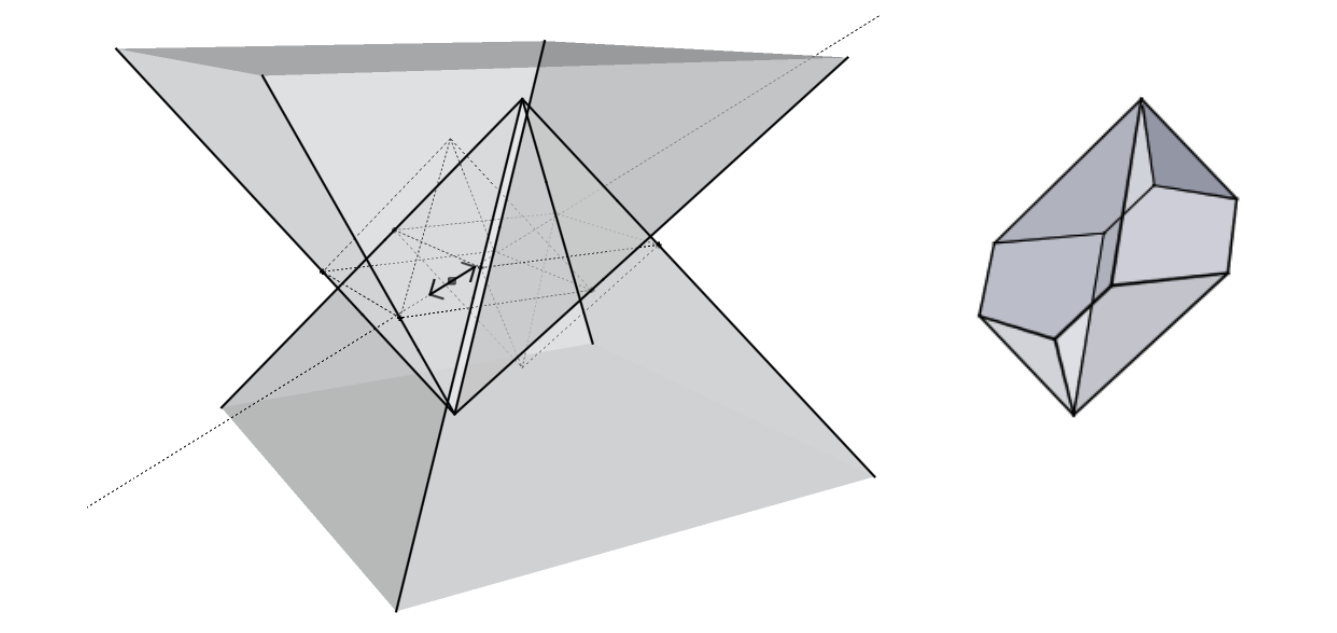}
\caption{}
\end{figure}

\smallskip\noindent\emph{Notice.} In the proof of Theorem \ref{diamondtodiamond} we only need the special case of Lemma \ref{cones}(a) when $C=D$ is a corner cone of $\Diamond_n$. But unlike the part (b), this part can be extended to arbitrary cones $C$ and $D$.

\begin{proof} (a) For $\kappa=\frac{x_2-x_3}{x_1-x_3}$ we have
$$
(x_2+C)\cap(x_3+D)=(1-\kappa)x_3+\kappa\big((x_1+C)\cap(x_3+D)\big).
$$

\medskip\noindent(b) Denote
\begin{align*}
\wedge^{(n)}=e_n+\RR_+(\Diamond_n-e_n),\qquad\vee^{(n)}=z-e_n+\RR_+(\Diamond_n+e_n).
\end{align*}
Pick a vertex
$$
v\in\vertex\big(\wedge^{(n)}\cap\big(z+\vee^{(n)}\big)\big)\notin\{e_n,z-e_n\}.
$$
There are positive dimensional faces $F\subset\wedge^{(n)}$ and $G\subset\vee^{(n)}$, uniquely determined by the condition
\begin{equation}\label{wedgevee}
v=\int(F)\cap(z+\int(G)).
\end{equation}

We claim $F\cap G\not=\emptyset$. In fact, if $F\cap G=\emptyset$ then there are facets $\tilde F\subset\wedge^{(n)}$ and $\tilde G\subset\vee^{(n)}$ which are centrally symmetric w.r.t. $0$ and such that $F\subset\tilde F$, $G\subset\tilde G$, and $\tilde F\cap\tilde G=\emptyset$ -- a general property of the faces of $\Diamond_n$. But this contradicts the condition $\dim(\Diamond_n\cap(z+\Diamond_n))=n$.

 The condition $F\cap G\not=\emptyset$, together with(\ref{wedgevee}), implies $\dim(F\cap G)=0$. (In particular, $F\cap G$ is a vertex of $\Diamond_{n-1}$.) Consequently,
$z\in\lin(F)+\lin(G)$. Assume $z=z_1+z_2$ for some $z_1\in\lin(F)$ and $z_2\in\lin(G)$. Using (\ref{wedgevee}) again, for all $\lambda$, sufficiently close to $1$, we can write
\begin{align*}
\int(F)\cap(\lambda&z+\int(G))=\int(F)\cap(\lambda z_1+\int(G))=\\
&(\lambda-1)z_1+\big(\int(F)\cap(z_1+\int(G))\big)=\\
&\qquad\qquad(\lambda-1)z_1+\big(\int(F)\cap(z+\int(G))\big)=\\
&\qquad\qquad\qquad\qquad(\lambda-1)z_1+v.
\end{align*}

In particular, the facets of $\wedge^{(n)}\cap(z+\vee^{(n)})$, meeting at the vertex $v$, and those of $\wedge^{(n)}\cap(\lambda z+\vee^{(n)})$, meeting at the vertex $(\lambda-1)z_1+v$, differ by the parallel translation by $(\lambda-1)z_1$. So the two corner cones are same.

But the corner cones at the vertices $e_n,\ \lambda z-e_n\in \vertex\big(\wedge^{(n)}\cap\big(\lambda z+\vee^{(n)}\big)\big)$
are also independent of $\lambda$. So, for a sufficiently small $\epsilon>0$, there is an affine function
\begin{align*}
\Theta:(1-\epsilon,1+\epsilon)\to(\RR^n)^N,\quad N=\#\vertex\big(\wedge^{(n)}\cap\big(z+\vee^{(n)}\big)\big),
\end{align*}
satisfying the conditions:

\smallskip\noindent$\centerdot$ $[\Theta(1)]=\vertex\big(\wedge^{(n)}\cap\big(z+\vee^{(n)}\big)\big)$,\ \ $[\Theta(\lambda)]\subset\vertex\big(\wedge^{(n)}\cap\big(\lambda z+\vee^{(n)}\big)\big)$,\ \ $\#[\Theta(1)]=\#[\Theta(\lambda)]$.

\smallskip\noindent$\centerdot$ for every element $[\Theta(1)]$ there is an element of $[\Theta(\lambda)]$ such that the corresponding corner cones are equal,

\smallskip\noindent where $[\Theta(-)]$ refers to the corresponding $N$-element subset of $\RR^n$.

Since $[\Theta(1)]$ is the complete vertex set of the polytope $\wedge^{(n)}\cap\big(z+\vee^{(n)}\big)$, the corner cones of $\wedge^{(n)}\cap\big(\lambda z+\vee^{(n)}\big)$ at the elements of $[\Theta(\lambda)]$ also form the complete set of corner cones of a polytope for every $\lambda\in(1-\epsilon,1+\epsilon)$, i.e.,
$$
[\Theta(\lambda)]=\vertex\big(\wedge^{(n)}\cap\big(\lambda z+\vee^{(n)}\big)\big).
$$
\end{proof}

\begin{lemma}\label{slidingdiamond}
For an element $z\in\RR^n$ with $\dim(\Diamond_n\cap(z+\Diamond_n))=n$, there exists a real number $\epsilon>0$ such that
$\cN\big(\Diamond_n\cap(\lambda z+\Diamond_n)\big)=\cN\big(\Diamond_n\cap(z+\Diamond_n)\big)$ for all $\lambda\in(1-\epsilon,1+\epsilon)$.
\end{lemma}

\begin{proof}
Consider the system of semi-open pyramids over $(n-1)$-dimensional crosspolytopes:
\begin{align*}
&\triangle^{(i)}=\conv\big(e_i,\{\pm e_j\}_{j\not=i}\big)\setminus\conv\big(\{\pm e_j\}_{j\not=i}\big),\\
&\bigtriangledown^{(i)}=\conv\big(-e_i,\{\pm e_j\}_{j\not=i}\big)\setminus\conv\big(\{\pm e_j\}_{j\not=i}\big),\\
&\qquad\qquad\qquad\qquad\qquad\qquad\qquad\qquad\quad\quad i=1,\ldots,n
\end{align*}

Since the bottoms of pyramids have been removed, for every index $i=1,\ldots,n$, we have the inclusion:
\begin{equation}\label{vertexunion1}
\begin{aligned}
\vertex\big(&\triangle^{(i)}\cap(z+\triangle^{(i)})\big)\ \cup\ \vertex\big(\triangle^{(i)}\cap(z+\bigtriangledown^{(i)})\big)\ \cup\\
&\vertex\big(\bigtriangledown^{(i)}\cap(z+\triangle^{(i)})\big)\ \cup\ \vertex\big(\bigtriangledown^{(i)}\cap(z+\bigtriangledown^{(i)})\big)\ \subset\ \vertex\big(\Diamond_n\cap(z+\Diamond_n)\big),\\
\end{aligned}
\end{equation}
where $\vertex(-)$ refers to the vertices of the corresponding topological closures, not in the bases of pyramids which have been removed, and the union is disjoint.

We can choose a real number $\delta>0$ such that
$$
\dim\big(\Diamond_n\cap(\lambda z+\Diamond_n)\big)=n,\qquad\lambda\in(1-\delta,1+\delta).
$$
For each such $\lambda$ and $i=1,\ldots,n$, consider the disjoint union
\begin{align*}
\vertex(\triangle^{(i)},&\bigtriangledown^{(i)},\lambda):=\\
&\vertex\big(\triangle^{(i)}\ \cap\ (\lambda z+\triangle^{(i)})\big)\cup\vertex\big(\triangle^{(i)}\ \cap\ (\lambda z+\bigtriangledown^{(i)})\big)\cup\\
&\quad\qquad\quad\vertex\big(\bigtriangledown^{(i)}\ \cap\ (\lambda z+\triangle^{(i)})\big)\ \cup\ \vertex\big(\bigtriangledown^{(i)}\cap(\lambda z+\bigtriangledown^{(i)})\big).
\end{align*}
Thus, the left hand side of (\ref{vertexunion1}) equals $\vertex(\triangle^{(i)},\bigtriangledown^{(i)},1)$.

By Lemma \ref{cones}, there are $0<\delta'\le\delta$ and affine functions
\begin{align*}
\Phi_i:(1-\delta',1+\delta')\to(\RR^n)^{N_i},\quad N_i=\#\vertex(\triangle^{(i)},\bigtriangledown^{(i)},1)&,\quad i=1,\ldots,n,
\end{align*}
such that for every $i$:

\medskip\noindent$\centerdot$ $[\Phi_i(1)]=\vertex(\triangle^{(i)},\bigtriangledown^{(i)},1)$,\ \ $[\Phi_i(\lambda)]\subset\vertex(\triangle^{(i)},\bigtriangledown^{(i)},\lambda)$,\ \ $\#[\Phi_i(1)]=\#[\Phi_i(\lambda)]$,

\smallskip\noindent$\centerdot$ for every element $[\Phi_i(1)]$ there is an element of $[\Phi_i(\lambda)]$ such that the corresponding corner cones are equal,

\medskip\noindent where: $[-]$ has the same meaning as in the proof of Lemma \ref{cones}(b) and `the corner cone at an element of $[\Phi_i(\lambda)]$' means the corresponding corner cone of the uniquely determined set from the following four possibilities:
$$
\big(\triangle^{(i)}\cap(\lambda z+\triangle^{(i)})\big),\quad \big(\triangle^{(i)}\cap(\lambda z+\bigtriangledown^{(i)})\big),\quad \big(\bigtriangledown^{(i)}\cap(\lambda z+\triangle^{(i)})\big),\quad \big(\bigtriangledown^{(i)}\cap(\lambda z+\bigtriangledown^{(i)})\big).
$$

\medskip\noindent\emph{Notice.} We have not excluded the possibility of the strict containment $[\Phi_i(\lambda)]\subsetneq\vertex(\triangle^{(i)},\bigtriangledown^{(i)},\lambda)$ for some $i$ and $\lambda$.

\medskip For every $\lambda\in(1-\delta',1+\delta')$ we have
$$
\begin{aligned}
\bigcup_{i=1}^n\bigg(\big(\triangle^{(i)}&\cap\big(\lambda z+\triangle^{(i)}\big)\big)\ \cup\ \big(\triangle^{(i)}\cap\big(\lambda z+\bigtriangledown^{(i)}\big)\big)\ \cup\\ &\big(\bigtriangledown^{(i)}\cap\big(\lambda z+\triangle^{(i)}\big)\big)\ \cup\ \big(\bigtriangledown^{(i)}\cap\big(\lambda z+\bigtriangledown^{(i)}\big)\big)\bigg)\ =\\
&\qquad\qquad\big(\Diamond_n\cap\big(\lambda z+\Diamond_n\big)\big)\setminus\bigcap_{i=1}^n\big(H^{(i)}\cup\big(\lambda z+H^{(i)}\big)\big),\\
\end{aligned}
$$
where, as in Lemma \ref{vertexunion}, $H^{(i)}\subset\RR^n$ denotes the $i$-th coordinate hyperplane.

\medskip The corner cone of $\Diamond_n\cap(z+\Diamond_n)$ at any vertex from
$$
\vertex\big(\Diamond_n\cap(z+\Diamond_n)\big)\bigcap\big(\int(\Diamond_n)\cup\int(z+\Diamond_n)\big)
$$
is either a corner cone of $\Diamond_n$ or $z+\Diamond_n$ and, therefore, the corresponding corner cone of $\Diamond_n\cap(\lambda z+\Diamond_n)$ is independent of $\lambda$.

Now Lemma \ref{vertexunion} implies the existence of a real number $\epsilon>0$ and an affine map
$$
\Psi:(1-\epsilon,1+\epsilon)\to(\RR^n)^M,\qquad M=\vertex\big(\Diamond_n\cap(z+\Diamond_n)\big),
$$
satisfying the conditions similar to those for the $\Phi_i$. Since $[\Psi(1)]$ is the complete vertex set of $\Diamond_n\cap(z+\Diamond_n)$, the corner cones of $\Diamond_n\cap(\lambda z+\Diamond_n)$ at the vertices from $[\Psi(\lambda)]\subset\vertex\big(\Diamond_n\cap(\lambda z+\Diamond_n)\big)$ also form the complete set of corner cones of a polytope  for every $\lambda\in(1-\epsilon,1+\epsilon)$. That is,
\begin{align*}
[\Psi(\lambda)]=\vertex\big(\Diamond_n\cap(\lambda z+\Diamond_n)\big),\qquad\lambda\in(1-\epsilon,1+\epsilon).
\end{align*}
\end{proof}

\begin{proof}[Ending of the proof of Theorem \ref{diamondtodiamond}]
Assume $f\in\vertex(\Diamond_m,\Diamond_n)$ and $z=f(0)\not=0$. By Corollary \ref{convexposition1}(b), $f(\vertex(\Diamond_m))\subset\vertex\big(\Diamond_n\cap(z+\Diamond_n)\big)$.
By Lemma \ref{slidingdiamond}, there is a real number $\epsilon>0$ and an affine map
$$
(1-\epsilon,1+\epsilon)\to\Hom(\Diamond_m,\Diamond_n),\qquad\lambda\mapsto f_\lambda,
$$
satisfying the conditions:

\smallskip\noindent$\centerdot$ $f_1=f$ and $f_\lambda(0)=\lambda z$,

\smallskip\noindent$\centerdot$ $f_\lambda(\vertex(\Diamond_m))\subset\vertex\big(\Diamond_n\cap(\lambda z+\Diamond)\big)$.

\smallskip But then $\{f_\lambda\}_{(1-\epsilon,1+\epsilon)}$ is an affine perturbation of $f$ (after rescaling $\lambda$ along the affine isomorphism $(1-\epsilon,1+\epsilon)\to(-1,1)$), a contradiction by Lemma \ref{perturbation}.
\end{proof}

\section{Partial result on $\vertex(\Box_m,\Diamond_n)$}\label{Boxtodiamond}

\begin{proposition}\label{boxtodiamond}
For all $m,n\ge2$ we have
\begin{align*}
\#\vertex(\Box_m,\Diamond_n)\ge 2n+2mn(2n-1)+2mn(m-1)(n-1).
\end{align*}
\end{proposition}

\begin{proof}
We have $\#\vertex^{(0)}(\Box_m,\Diamond_n)=\#\vertex(\Diamond_n)=2n$. By Proposition \ref{calculus1}(c), $\#\vertex^{(1)}(\Box_m,\Diamond_n)=2m{2n\choose2}=2mn(2n-1)$.

There are ${m\choose2}$ orthogonal projections of $\rho:\Box_m\to\Box_2$ along the codimension 2 faces of $\Box_m$, there are ${n\choose 2}$ isometric embeddings $\Diamond_2\to\Diamond_n$, and there are $8$ affine isomorphisms $\kappa:\Box_2\to\Diamond_2$. We have $8{m\choose2}{n\choose2}=2mn(m-1)(n-1)$ different rank 2 maps $\pi\kappa\rho:\Box_m\to\Diamond_n$, each mapping $\vertex(\Box_m)$ to $\vertex(\Diamond_n)$. By Proposition \ref{calculus1}(a), all these $\pi\kappa\rho$-s belong to $\vertex^{(2)}(\Box_m,\Diamond_n)$.
\end{proof}

The estimate in Proposition \ref{boxtodiamond} is far from optimal: \textsf{Polymake} computations yield $\#\vertex(\Box_3,\Diamond_4)=27968$, whereas the right hand side of the inequality in the proposition is just $316$. Currently we do not even have a conjectural description of the vertices of $\Hom(\Box_m,\triangle_n))$.

\bigskip\noindent\emph{Acknowledgement.} We thank Brian Cruz for computing $\beta(5)$ and the anonymous referee whose comments helped improving the paper.

\bibliography{references}
\bibliographystyle{plain}

\end{document}